\documentclass[english,11pt,reqno]{amsart}

\usepackage[dvipsnames,svgnames,x11names,hyperref]{xcolor}
\usepackage{latexsym,amsmath,amssymb}
\usepackage[applemac]{inputenc}
\usepackage[all,arc,curve,frame]{xy}
\xyoption{matrix}
\usepackage{multirow,ctable}

\usepackage[shortlabels]{enumitem}
\setlist[enumerate]{label=\rm{(\arabic*)}}
\setlist[enumerate,2]{label=\rm({\it\roman*})}
\setlist[itemize]{label=\raisebox{0.25ex}{\tiny$\bullet$}}
\usepackage[normalem]{ulem} 
\usepackage{lmodern} 
\usepackage{lscape}

\usepackage[colorlinks, linktocpage, pdfauthor={J. Blanc \& S. Lamy},
citecolor = Navy, linkcolor = Navy]{hyperref}

\newtheorem{lemm}{Lemma}[section]
\newtheorem{prop}[lemm]{Proposition}

\newtheorem{coro}[lemm]{Corollary}
\newtheorem{theo}[lemm]{Theorem}

\theoremstyle{remark}
\newtheorem{rem}[lemm]{Remark}

\newcommand\dashmapsto{\mapstochar\dashrightarrow}

\newcommand{\C}{\mathbb{C}}
\newcommand{\A}{\mathbb{A}}

\newcommand{\p}{\mathbb{P}}

\newcommand{\eps}{{\varepsilon}}
\newcommand{\Bir}{\mathrm{Bir}}
\renewcommand{\phi}{\varphi}

\newcommand{\Lp}{\mathcal{L}_{\text{\rm plane}}}
\newcommand{\Lq}{\mathcal{L}_{\text{\rm quadric}}}

\DeclareMathOperator{\bir}{Bir}

\DeclareMathOperator{\PGL}{PGL}

\DeclareMathOperator{\Pic}{Pic}
\DeclareMathOperator{\NE}{NE}

\newcolumntype{L}{>{$}l<{$}}
\newcolumntype{C}{>{$}c<{$}}

\title{On birational maps from cubic threefolds}
\date{\today}
\author{J\'er\'emy Blanc}
\address{Mathematisches Institut \\ 
Universit\"at Basel \\
Rheinsprung 21 \\
CH-4051 Basel \\
Switzerland}
\email{jeremy.blanc@unibas.ch}

\author{St\'ephane Lamy}
\address{Institut de Math\'ematiques de Toulouse,
Universit\'e Paul Sabatier,
118 route de Narbonne,
31062 Toulouse Cedex 9, France}
\email{slamy@math.univ-toulouse.fr}
\thanks{The authors gratefully acknowledge support by the Swiss National Science Foundation Grant  ``Birational Geometry'' PP00P2\_128422 /1 and by the French National Research Agency Grant ``BirPol'', ANR-11-JS01-004-01.}

\begin{document}

\begin{abstract}
We characterise smooth curves in a smooth cubic threefold whose blow-ups produce a weak-Fano threefold.
These are curves $C$ of genus~$g$ and degree~$d$, such that (i) $2(d-5) \le g$ and $d\le 6$; (ii) $C$ does not admit a 3-secant line in the cubic threefold.
Among the list of ten possible such types $(g,d)$, two yield Sarkisov links that are birational selfmaps of the cubic threefold, namely $(g,d) = (0,5)$ and $(2,6)$.  
Using the link associated with a curve of type $(2,6)$, we are able to produce the first example of a pseudo-automorphism with dynamical degree greater than~$1$ on a smooth threefold with Picard number~$3$.
We also prove that the group of birational selfmaps of any smooth cubic threefold contains elements contracting surfaces birational to any given ruled surface.
\end{abstract}
\maketitle

\section*{Introduction}

The two archetypal examples of Fano threefolds that are not rational are smooth cubics and quartics in $\p^4$. 
The proofs go back to the early 1970', and are quite different in nature.
The non-rationality of a smooth quartic $X$ was proved by Iskovskikh and Manin \cite{IM} by studying the group $\Bir(X)$ of birational selfmaps of $X$: They show that this group is equal to the automorphism group of $X$, which is finite.
On the other hand, the proof of the non-rationality of a smooth cubic $Y$ by Clemens and Griffiths \cite{CG} relies on the study of the intermediate Jacobian of such a threefold: They prove that the intermediate Jacobian of~$Y$ is not a direct sum of Jacobian of curves, as it would be the case if $Y$ were rational.
Once we know that a smooth cubic threefold is not rational one can deduce some interesting consequences about the group $\bir(Y)$. 
For instance, the non-rationality but unirationality of the cubic $Y$ implies that it does not admit any non-trivial algebraic $\mathbb G_a$-action (see Proposition~\ref{Prop:Daigle}).
It is tempting to try to go the other way around: study the group $\Bir(Y)$ first with the aim of pointing out qualitative differences with the group $\Bir(\p^3)$, hence obtaining an alternative proof of the non-rationality.
This was one of the motivations of the present work.
However we must admit that the answers we get indicate that the group $\Bir(Y)$ is in many respects quite as complicated as $\Bir(\p^3)$, and it is not clear whether the above strategy could be successful.   

A first indication that the group $\Bir(Y)$ is quite large comes from the generalisation  of the Geiser involution of $\p^2$.
Let us recall this classical construction.
Consider $p \in Y$ a point. 
A general line of $\p^4$ through $p$ intersects $Y$ in two other points, and one can define a birational involution of~$Y$ by exchanging any such two points.
One can see this involution as a Sarkisov link: first blow-up $Y$ to obtain a weak-Fano threefold $X$, then flop the strict transforms of the 6 lines through $p$, and then blow-down the transform of the tangent hyperplane section at $p$ to come back to $Y$.
It is natural to try to obtain new examples by blowing-up curves instead of a point.

The general setting of this paper (where we always work over the field $\C$ of complex numbers) is to consider $Y\subset \p^4$  a smooth cubic threefold, $C\subset Y$  a smooth (irreducible) curve of genus~$g$ and degree~$d$, and $\pi\colon X \to Y$ the blow-up of $C$.
We are interested in classifying the pairs $(g,d)$ such that $X$ is (always, or 
generically) weak-Fano, that is, the anticanonical divisor $-K_X$ is big and nef.
To express our main result it is convenient to introduce the following lists of
pairs of genus and degree:
\begin{align*}
\Lp &= \{\, (0,1),\, (0,2),\, (0,3),\, (1,3),\, (1,4),\, (4,6)\, \}; \\
\Lq &= \{\, (0,4),\, (0,5),\, (1,5),\, (2,6)\, \}. 
\end{align*}
We shall see (Lemma~\ref{Lemm:Smalldegreecurves} and Proposition~\ref{Prop:12 possibilities}) that the 12 pairs in $\Lp \cup \Lq\cup \{(2,5),(3,6)\}$ correspond to all $(g,d)$ with $2(d-5) \le g$ and $d \le 6$ such that there exists a smooth curve $C\subset Y$ of type $(g,d)$, that is, of genus~$g$ and degree~$d$. However, curves of type $(2,5)$ or $(3,6)$ always admit a $3$-secant line in $Y$ (Proposition~\ref{Prop:NecessaryConditionsPropre}), which is an obvious obstruction for $X$ to be weak-Fano.

We obtain a result similar to the case of $\p^2$ or $\p^3$ (see \cite{BL}): the blow-up of a smooth curve $C\subset Y$ of type $(g,d)$ is weak-Fano if and only  $2(d-5) \le g$, $d \le 6$ and there is no $3$-secant line to $C$. More precisely, we obtain the following:
\renewcommand\thelemm{\Alph{lemm}}
\begin{theo}\label{Thm:WeakFano}
Let $Y\subset \p^4$ be a smooth cubic threefold, let $C\subset Y$ be a smooth curve of genus~$g$ and degree~$d$, and denote by $X\to Y$ the blow-up of~$C$.
Then:
\begin{enumerate}[leftmargin=*]
\item \label{case:thm1} If $X$ is a weak-Fano threefold, then $C$ is contained in a smooth hyperquadric section, $\lvert -K_X\rvert$ is base-point-free, $(g,d) \in \Lp \cup \Lq$ and there is no $3$-secant line to $C$ in $Y$. Moreover, $C$ is contained in a hyperplane section if and only if $(g,d)\in \Lp$.
\item \label{case:thm2} Conversely :
\begin{enumerate}
\item \label{case:thmi} If $(g,d) \in \Lp$ then $X$ is weak-Fano and more precisely:
\begin{itemize}[leftmargin=*]
\item If $(g,d) = (0,1)$ or $(1,3)$, then $X$ is Fano;
\item If $(g,d) = (1,4)$ or $(4,6)$, then $X$ is weak-Fano with divisorial anticanonical morphism; 
\item If $(g,d) = (0,2)$ or $(0,3)$, then $X$ is weak-Fano with small anticanonical morphism;
\end{itemize}
\item \label{case:thmii} If $(g,d) \in \Lq$ and $C$ is a curve without any $3$-secant line in $Y$, then $X$ is weak-Fano with small anticanonical morphism.

Moreover for these four cases, there exists a dense open set of such curves in the Hilbert scheme parametrising smooth curves of genus~$g$ and degree~$d$ in $Y$. 
\end{enumerate}
\end{enumerate}
\end{theo}

We refer the reader to the introduction of \cite{BL} for more information on Sarkisov links and the classification of weak-Fano threefolds of Picard number~$2$.
Let us mention that among the cases covered by Theorem~\ref{Thm:WeakFano},
it turns out that two yield Sarkisov links from the cubic to itself, namely $(g,d) = (0,5)$ and $(2,6)$: this follows from the exhaustive lists computed in the paper \cite{CM}.
The case $(2,6)$ is of particular interest.
First, it gives an example of an element of $\bir(Y)$ that contracts a non-rational ruled surface. 
This lead us to the question whether there were restrictions on the birational type of surfaces contracted by an element of $\Bir(Y)$, and quite surprisingly the answer is that, as in the case of $\p^3$, there is no such obstruction.

\begin{prop} \label{Prop:mainB}
Let $Y\subset \p^3$ be a smooth cubic hypersurface, and let $\Gamma$ be an abstract irreducible curve.
Then, there exists a birational map in $\Bir(Y)$ that contracts a surface birational to $\Gamma\times \p^1$.
\end{prop}

A second interesting feature of curves of type $(2,6)$ is that they come in pair: there exist pencils of hyperquadric sections on $Y$ whose base locus is the union of two curves of type $(2,6)$. 
To each of these curves is associated an involution, and the composition of the two involutions yields an interesting map from the dynamical point of view.

\begin{prop} \label{Prop:mainC}
Let $Y \subset \p^4$ be a smooth cubic threefold.
There exists a pencil $\Lambda$ of hyperquadric sections on $Y$, whose base locus is the union of two smooth curves $C_1, C_2$ of genus $2$ and degree $6$, and such that there exists a pseudo-automorphism of dynamical degree equal to $49+20\sqrt{6}$ on the threefold~$Z$ obtained from $Y$ by blowing-up successively $C_1$ and $C_2$. 
\end{prop}
\renewcommand\thelemm{\thesection.\arabic{lemm}}

Let us end this introduction by discussing the similarities and differences with the  papers \cite{BL} and \cite{ACM, CM}. 
In \cite{ACM} the authors show the existence of weak-Fano threefolds obtained as blow-ups of a smooth curve on a smooth threefold $Y$, where $Y$ is either a quadric in $\p^4$, a complete intersection of two quadrics in $\p^5$, or the Fano threefold $V_5$ of degree 5 in~$\p^6$. 
The construction relies on the fact that for each curve $C$ of genus~$g$ and degree~$d$ candidate to be blown-up, there exists a smooth K3 surface $S \subset Y$ containing a curve of type $(g,d)$, such that the Picard group of $S$ is generated by the curve and a hyperplane section. 
This allows to control the geometry of potential bad curves (curves intersecting many times $C$ with respect to their degree).
In fact one can apply a similar strategy to the case of a cubic threefold: This is done for the type $(0,5)$ and $(2,6)$ in \cite[Proposition~2.10]{CM}. 
Precisely J. Cutrone and N. Marshburn show that there exist very special curves of type $(0,5)$ or $(2,6)$ that are contained in a particular K3 surface, which is itself contained in a particular smooth cubic threefold, such that the blow-up is weak-Fano.
By contrast, we show that for any cubic threefold and any curve of type $(0,5)$ or $(2,6)$ with no $3$-secant line in the cubic (which is an open and dense condition), the blow-up is weak-Fano. 
In short, existence was already known, but now we obtain genericity (and a different proof for existence).

In many respect we follow a similar line of argument as in \cite{BL}, but with some notable simplification (it would be possible to implement these simplifications to the case of blow-ups of $\p^3$ as well).
In particular in \cite[Proposition 2.8]{BL} we proved and used the fact that if the blow-up of a curve $C \subset \p^3$ gives rises to a weak-Fano threefold, then it is contained in a smooth quartic.
The analogous statement in the context of a cubic threefold~$Y$ would be that such a curve $C$ is contained in a smooth hyperquadric section.
We avoid such a statement (true, but quite delicate to prove), and we prove instead by more elementary arguments that curves of small degree are always contained in such a smooth hyperquadric section, without assuming a priori that the blow-up is weak-Fano.\\

The paper is organised as follows.

In Section~\ref{Sec:NecessarySection}, we develop tools to show that if the blow-up of a curve in a smooth threefold is weak-Fano, the type of the curve is in $\Lp\cup \Lq$, and belongs to a hyperplane section if and only if it is in $\Lp$ (see Proposition~\ref{Prop:NecessaryConditionsPropre}). 
This is done first by recalling basic facts in $\S\ref{Sec:Preliminaries}$, then describing curves of small degree in cubic threefolds in $\S\ref{Sec:small_degree}$, and showing that the curves that yield weak-Fano threefolds are of small degree in $\S\ref{Sec:WeakFanoSmalldegree}$. The case of curves in hyperplane section is studied in $\S\ref{Sec:hyperplane}$, and the proof of Proposition~\ref{Prop:NecessaryConditionsPropre} is then achieved in $\S\ref{Sec:SummaryNec}$.

In Section~\ref{Sec:ProofofTheoremA}, we do the converse. We take a curve of type in $\Lp\cup \Lq$, assume that it is not in a hyperplane section if $(g,d)\in\Lq$, and moreover that it has no $3$-secant line if it is of type~$(0,5)$. 
We then show that the blow-up is weak-Fano, by proving first that the curve lies inside a nice surface (in most cases, a  Del Pezzo surface of degree $4$), and then that it is contained in a smooth hyperquadric section ($\S\ref{Sec:FanoCase}$ and $\S\ref{Sec:smooth_sections}$). The fact that the open subset corresponding to curves without any $3$-secant line is not empty is established in $\S\ref{Sec:hyperquadric}$, by constructing examples in special singular hyperquadric sections. This leads to the proof of Theorem~\ref{Thm:WeakFano} ($\S\ref{Sec:SummaryProof}$), that we complement with a description of the Sarkisov links associated with the blow-ups of such curves ($\S\ref{Sec:links}$).

Section~\ref{Sec:Examples} is devoted to the proof of Propositions~\ref{Prop:mainB} and~\ref{Prop:mainC}. 
After recalling some basic notions in \S\ref{Sec:BasicOfLastSec}, we prove the existence of birational transformations of smooth cubic threefolds of arbitrary genus in $\S\ref{Sec:Genus}$, and finish our article with $\S\ref{Sec:Pseudoauto}$, which describes the construction of pseudo-automorphisms associated with the curves of type $(2,6)$.\\

The authors thank Yuri Prokhorov for interesting discussions during the preparation of the article.

\section{Curves leading to weak-Fano threefolds belong to the lists}\label{Sec:NecessarySection}

In the next sections (\S\ref{Sec:Preliminaries}--\S\ref{Sec:hyperplane}), we give conditions on a smooth curve $C\subset Y$ in a smooth cubic threefold that are necessary for the blow-up of~$Y$ along~$C$ to be weak-Fano. 
These will be summarised in $\S\ref{Sec:SummaryNec}$: 
We prove in Proposition~\ref{Prop:NecessaryConditionsPropre} that $(g,d)\in \Lp$ if $C$ is contained in a hyperplane section,
and that $(g,d)\in \Lq$ if $C$ not contained in a hyperplane section.
\subsection{Preliminaries}\label{Sec:Preliminaries}
This section is devoted to reminders of some results that we shall need in the sequel.

Let $Z$ be a smooth projective variety (we have in mind $Z = \p^n$, or $Z \subset \p^4$ a cubic hypersurface), and let $C \subset Z$ be a smooth curve.
Let $X \to Z$ be the blow-up of $C$ in $Z$, with exceptional divisor $E$.
We say that another (irreducible) curve $\Gamma \subset Z$ is $n$-secant to $C$ if the strict transform $\tilde \Gamma$ of $\Gamma$ in $X$ satisfies $E \cdot \tilde \Gamma \ge n$. 
The multiplicity of $\Gamma$ as a $n$-secant curve to $C$ is the binomial coefficient $\left( \begin{smallmatrix}
E \cdot \tilde{\Gamma} \\ n
\end{smallmatrix} \right)$.
The following is a basic observation.

\begin{lemm}\label{Lemm:fano1}
Let $C \subset Y \subset \p^4$ be a smooth curve in a smooth cubic threefold, and consider $\pi\colon X \to Y$ the blow-up of $C$.
If $\Gamma$ is a curve of degree $n$ which is $m$-secant to $C$, then 
$$-K_X \cdot \tilde \Gamma \le 2n -m.$$
\end{lemm}

\begin{proof}
The result follows from the ramification formula $K_X = \pi^* K_Y + E$, together with $-K_Y \sim 2H$ where $H$ is a hyperplane section. 
\end{proof}

Observe in particular that if $C$ admits a 2-secant line (resp.~a 3-secant line), then the blow-up $X$ will not be Fano (resp.~weak-Fano).
We mention now two results from \cite{HRS} about $n$-secant lines. 
The first one is analogue to the classical formula of Cayley that counts the number of $4$-secant lines to a curve in $\p^3$. 
The second one is about rational quintic curves, which will prove to be one of the most delicate cases in our study.

\begin{lemm}[{\cite[Lemma 4.2]{HRS}}]\label{Lemm:HRS4.2}
Let $C \subset Y \subset \p^4$ be a smooth curve of genus~$g$ and degree~$d$ in a smooth cubic threefold.
Let 
$$N = \frac{5d(d-3)}2 +6-6g. $$ 
If $C$ does not admit infinitely many $2$-secant lines in $Y$, and $N \ge 0$, then $N$ is the number of $2$-secant lines to $C$ in $Y$, counted with multiplicity.  
\end{lemm}

We should mention that in the case of a line the previous lemma gives $N = 1$, which does not fit well with our explicit definition of $n$-secant and multiplicity. 
One can take it as a convention that in this case the line should be considered as a $2$-secant line to itself with multiplicity 1.

\begin{lemm}[{\cite[Corollary 9.3]{HRS}}]\label{Lemm:HRS9.3}
Let $C \subset \p^4$ be a smooth rational quintic curve, and assume that $C$ is not contained in a hyperplane.
Then there exists a unique $3$-secant line to $C$ in $\p^4$.
\end{lemm}

The following lemma is classical (see \cite[Lemma 2.2.14]{IP}), and gives the first basic inequality on the pair $(g,d)$. 
It is analogue to the fact that the blow-up of at least $9$ points in $\p^2$ is never weak-Fano.

\begin{lemm}
\label{Lemm:K^3}
Let $C \subset Y$ be a smooth curve of genus~$g$ in a smooth threefold, and let $\pi\colon X\to Y$ be the blow-up of $C$. Denote by $E$ the exceptional divisor.
Then
\begin{align*}
K_X^2\cdot E &= -K_Y\cdot C + 2 - 2g;\\
(-K_X)^3 &= (-K_Y)^3 +2K_Y\cdot C - 2 + 2g.
\end{align*}
In particular, if $Y \subset \p^4$ is a smooth cubic and $C\subset Y$ is a smooth curve of genus~$g$ and degree~$d$, then
\begin{align*}
K_X^2\cdot E &= 2+2d-2g;\\
(-K_X)^3 &=22-4d+2g;
\end{align*}
and $(-K_X)^3 > 0 \Longleftrightarrow 2(d - 5) \le g $. 
\end{lemm}

We shall use the following direct consequence of the Riemann-Roch formula, in order to decide when the curve is contained in a hyperplane or a hyperquadric.

\begin{lemm} \label{Lemm:hypersurfaces}
Let $C \subset \p^4$ be a smooth curve of genus~$g$ and degree~$d$.
\begin{enumerate}
\item \label{item:plane1}
If $d > 2g-2$, then the projective dimension of the linear system of hyperplanes that contain $C$ is at least equal to $3-d +g$.
\item \label{item:plane2}
If $2g-4\ge d \ge 2g-2$ and $g\ge 2$, then the projective dimension of the linear system of hyperplanes that contain $C$ is at least equal to $2-d +g$.
\item \label{item:planelist}
In particular, if $C\subset \p^4$ is a smooth curve of genus~$g$ and degree~$d$, with $(g,d) \in \Lp \cup \{\, (2,5), (3,6)\, \}$ 
then $C$ is contained in a hyperplane section.
\item \label{item:quadric}
If $d > g-1$, then the projective dimension of the linear system of quadrics that contain $C$ is at least equal to $13-2d +g$.
In particular if $(g,d) \in \Lq$, then $C$ is contained in a pencil of hyperquadric sections.
\end{enumerate}
\end{lemm}

\begin{proof}
Let $m\in \{1,2\}$ and let $D$ be the divisor of a hyperplane section restricted to $C$.
By the Riemann-Roch formula
$$\ell(mD) - \ell(K_C-mD) = md+1-g.$$
In cases~\ref{item:plane1} and~\ref{item:quadric}, by assumption $md > 2g-2$ hence $\ell(K_C-mD) = 0$ and $\ell(D) = md -g+1$.
But the vectorial dimensions of the systems of hyperplanes and quadrics in $\p^4$ are respectively $5$ and $15$, hence the vectorial dimensions of the systems of hyperplanes or quadrics containing $C$ are respectively $4 - d + g$ and $14 - 2d + g$.

In case~\ref{item:plane2}, if $\ell(K_C-D)\le 1$, then $\ell(D)\le d+2-g$.  
Since the vectorial dimension of the system of hyperplanes is $5$, the vectorial dimension of the system of hyperplanes containing $C$ is then at least $3-d+g$.
It remains to show that $\ell(K_C-D)\ge 2$ is not possible. The degree of the divisor $K_C-D$ is equal to $2g-2-d$, which belongs to $\{0,1,2\}$ by hypothesis. Since $g\ge 2$, the only possibility is then that $\deg(K_C-D)=\ell(K_C-D)=2$ and that $\lvert K_C-D\rvert$ induces a double covering $C\to \p^1$; in particular, $C$ is hyperelliptic. The divisor that yields the double covering is the unique $g^1_2$ and $K_C$ is linearly equivalent to $(g-1)\cdot g^1_2$ \cite[Proposition 5.3]{Har}, so $D\sim (g-2)\cdot g_1^2$. This is impossible because $D$ is very ample and $K_C$ is not very ample \cite[Proposition 5.2]{Har}.
\end{proof}

\subsection{Curves of small degree in $\mathbb{P}^4$} \label{Sec:small_degree}
In this section, we give some conditions on the genus and degree of a curve in $\mathbb{P}^4$, and then apply this to the case of curves of small degree that lie in smooth cubic threefolds.\\

We first recall the classical Castelnuovo's bound on non-degenerate smooth curves of $\p^n$ (\cite{Cas}, see also \cite[p.252]{GH}):

\begin{prop}\label{Prop:CastelnuovoBound}
Let $C\subset \p^n$ be a smooth curve of genus~$g$ and degree~$d$ which is not 
contained in a hyperplane, and write the Euclidean division 
$$d-1=(n-1)m+\eps,$$ where $m\ge 0$ and $\eps\in \{0,\dots,n-1\}$. Then 
$$g\le (n-1)m(m-1)/2+m\eps.$$
\end{prop}

Reducing to the cases of $n=3,4$ we obtain the following:

\begin{coro}\item\label{Coro:Castelnuovo34}
\begin{enumerate}
\item If $C\subset \p^3$ be a smooth curve of genus~$g$ and degree~$d$ which is not contained in a plane, then 
$$g\le \left\lfloor\frac{d^2}{4}\right\rfloor-d+1.$$
\item If $C\subset \p^4$ be a smooth curve of genus~$g$ and degree~$d$ which is not contained in a hyperplane, then we have the better bound
$$g\le \left\lfloor\frac{d(d-5)}{6}\right\rfloor+1 \le \left\lfloor\frac{d^2}{4}\right\rfloor-d+1.$$
\end{enumerate}
\end{coro}

\begin{proof}
\begin{enumerate}[wide]
\item Applying Proposition~\ref{Prop:CastelnuovoBound} to the case $n=3$ we obtain $d-1=2m+\eps $ and $g\le m(m-1+\eps)$. It remains to see that $m(m-1+\eps)=\lfloor\frac{d^2}{4}\rfloor-d+1$, which follows from the following equality
$$\frac{d^2}{4}-d+1=\frac{(2m+\eps+1)(2m+\eps-3)+4}{4}=m(m-1+\eps)+\frac{(1-\eps)^2}{4}.$$

\item Similarly, we apply Proposition~\ref{Prop:CastelnuovoBound} to the case $n=4$ and obtain $d-1=3m+\eps $ and $g\le 3m(m-1)/2+m\eps$. We then observe that $3m(m-1)/2+m\eps$ is equal to $\lfloor\frac{d(d-5)}{6}\rfloor+1$:
\[\begin{array}{rcl}
\displaystyle{\frac{d(d-5)}{6}+1}&=&\displaystyle{\frac{(3m+\eps+1)(3m+\eps-4)+6}{6}} \\
&=& \displaystyle{\frac{3m(m-1)}{2}+m\eps+\frac{(1-\eps)(2-\eps)}{6}}.\qedhere\end{array} \]
\end{enumerate}
\end{proof}

Using Castelnuovo's bound, we now obtain all possible values for the genus of a curve of small degree $d\le 6$ in a smooth cubic threefold:

\begin{lemm}\label{Lemm:Smalldegreecurves}
Let $Y\subset \p^4$ be a smooth cubic hypersurface.
\begin{enumerate}[wide]
\item \label{case1:smalldegree} If $C\subset Y$ is a smooth curve of degree $d\le 6$ and genus~$g$, then $g\le \tau(d)$, where $\tau(d)$ is given in the following table.\\

\begin{center}
\begin{tabular}{CCCCCCC}
\toprule
d & 1 & 2 & 3 & 4 & 5 & 6  \\
\tau(d) & 0 & 0 & 1 & 1 & 2 & 4 \\
\bottomrule
\end{tabular}
\end{center}
In particular, $(g,d)\in \Lp \cup  \Lq\cup \{(2,5),(0,6),(1,6),(3,6)\}$.
\item \label{case2:smalldegree} For each pair $(g,d)$ such that $1\le d\le 6$ and $0\le g \le \tau(d)$, and for each smooth hyperplane section $S\subset Y$, there exists a smooth curve $C\subset S\subset Y$ of genus~$g$ and degree~$d$.
\end{enumerate}
\end{lemm}

\begin{proof}
Suppose first that $C$ is contained in a plane. Since $Y$ does not contain any plane (a classical fact which can be easily checked in coordinates), then~$C$ is of degree at most $3$, hence $(g,d)\in \{(0,1),(0,2),(1,3)\}$. 

Assume now that $C$ is not contained in a plane and that $d\le 6$. 
By Proposition~\ref{Prop:CastelnuovoBound} we get $g\le \lfloor\frac{d^2}{4}\rfloor-d+1\le \tau(d).$
This achieves to prove Assertion \ref{case1:smalldegree}.\\

In order to prove \ref{case2:smalldegree}, we view the smooth cubic surface $S$ as the blow-up of $\p^2$ at six points $p_1,\dots,p_6$. The Picard group of $S$ is generated by $L,E_1,\dots,E_6$, where $L$ is the pull-back of a general line and $E_i$ is the exceptional curve contracted onto $p_i$ for each $i$. We choose a curve $D\subset\p^2$ of degree $k$ having multiplicity $m_i$ at $p_i$, so that the strict transform $C\subset S$, equivalent to $\tilde{C}\sim kL-\sum m_iE_i$, is smooth. The hyperplane section being equivalent to $-K_S=3L-\sum E_i$, the degree and genus of $C$ are equal to $$\begin{array}{ll}
g=\frac{(k-1)(k-2)}{2}-\sum \frac{m_i(m_i-1)}{2},&\ d=3k-\sum m_i.\end{array}$$ 

It is then an easy exercise to produce the desired curves:

We choose $D$ to be a conic through $l$ of the $p_i$, with $0\le l\le 5$ and obtain pairs $(0,d)$ with $1\le d\le 6$. 

With a smooth cubic passing through $l$ of the $p_i$, $0\le l\le 6$ we obtain pairs $(1,d)$ with $3\le d\le 9$.

Taking a quartic having a double point at $p_1$ and passing through $l$ other $p_i$, $0\le l\le 5$, we obtain pairs $(2,d)$ with $5\le d\le 10$.

Smooth quartics yield pairs $(3,d)$ with $6\le d\le 12$.

The pair $(4,6)$ is the complete intersection of $S$ with a general quadric, and corresponds to sextics with $6$ double points at the $p_i$.
\end{proof}

\subsection{Curves that yield a weak-Fano threefold have small degree}\label{Sec:WeakFanoSmalldegree}
In this section, we prove that if a smooth curve $C$ of genus~$g$ and degree~$d$ gives rise to a weak-Fano threefold after blow-up, then $(g,d)$ must belong to the list $\Lp\cup \Lq\cup \{(2,5),(3,6)\}$. 
The two last possibilities will be removed in the next section (Proposition~\ref{Prop:Always3secant}).\\

We shall need the following classical (but difficult) result about the anticanonical linear system on a weak-Fano threefold.

\begin{prop} \label{Prop:-K irreducible}
Let $X$ be a smooth weak-Fano threefold.
Then 
$$\dim \lvert -K_X \rvert = \frac12 (-K_X)^3 + 2 \ge 3$$
and the general member of $\lvert -K_X \rvert$ is an irreducible $K3$-surface $($in particular the base locus of $\lvert -K_X \rvert$ has at most dimension $1)$.
\end{prop}

\begin{proof}
The first assertion is a direct consequence of the Riemann-Roch formula for threefolds (see e.g.~\cite[p.~437]{Har}) and Kawamata-Viehweg vanishing (\cite[Theorem~4.3.1]{Laz1}).
For the second assertion, see \cite[Theorem~(0.4)]{shin}.
\end{proof}

\begin{prop} \label{Prop:12 possibilities}
Let $C \subset Y$ be a smooth curve of genus~$g$ and degree~$d$ in a smooth cubic threefold.
Assume that the blow-up $X$ of~$Y$ along $C$ is a weak-Fano threefold.
Then $d \le 6$, and  $(g,d)$ is given in the following table\\
\begin{center}
\upshape\begin{tabular}{cCCCCCCC}
\toprule
$d$ & 1 & 2 & 3 & 4 & 5 & 6\\
$g$ & 0 & 0 & 0,1 & 0,1& 0,1,2 &2,3,4\\
\bottomrule
\end{tabular}
\end{center}
\bigskip 
In other words, with the notation of the introduction, 
$$(g,d) \in \Lp\cup \Lq\cup \{(2,5),(3,6)\}.$$
\end{prop}

\begin{proof}
The linear system $\lvert-K_X\rvert$ corresponds on $Y$ to hyperquadric sections through $C$.
By Proposition~\ref{Prop:-K irreducible}, the curve $C$ is contained in a 2-dimensional linear system of hyperquadric sections, which has no fixed component, hence $d \le 8$.
Since a general hyperquadric section through $C$ is irreducible, if $C$ is also contained in a hyperplane section, then we immediately get $d \le 6$.
Otherwise, by Corollary~\ref{Coro:Castelnuovo34} and Lemma~\ref{Lemm:K^3} we have 
$$2(d-5) \le g \le \left\lfloor\frac{d(d-5)}{6}\right\rfloor+1,$$
which yields 

$$0\le \frac{d(d-5)}{6}+1-2(d-5)=\frac{(d-6)(d-11)}{6}.$$
Hence, $d\in \{7,8\}$ is not possible.
We get the list of possible $(g,d)$ from Lemma~\ref{Lemm:Smalldegreecurves}, by removing the cases $(0,6)$ and $(1,6)$, which do not satisfy the condition $2(d-5) \le g$. 
\end{proof}

\subsection{Curves contained in a hyperplane section}\label{Sec:hyperplane}
In this section, we show that if the blow-up of a smooth cubic threefold $Y$ along a smooth curve $C\subset Y$ of type $(g,d)$ is weak-Fano, and $C$ is contained in a hyperplane section, then $(g,d)\in \Lp$. This implies that the possibilities for $(g,d)$, without the assumption to be in a hyperplane, are $\Lp\cup \Lq$.\\

Note that a hyperplane section is an irreducible and reduced cubic surface~$S$, which is smooth in general but can also have singularities. For instance it can have double points or be the cone over a smooth cubic. Lemma~\ref{Lemm:smoothcubics} treats the case of smooth hyperplane sections and is used for the general case, treated in Proposition~\ref{Prop:Always3secant}.

\begin{lemm}
\label{Lemm:smoothcubics}
Let $C\subset  S\subset \p^3$ be a smooth curve of genus~$g$ and degree~$d$ in a smooth cubic surface $S$. If there is no $3$-secant line of $C$ to $S$, then 
$$(g,d)\in \{\, (0,1),\, (0,2),\, (0,3),\, (1,3),\, (1,4),\, (4,6)\, \}=\Lp.$$
Moreover, all possibilities occurs in each smooth cubic surface.
\end{lemm}

\begin{proof}
Choosing six disjoint lines $E_1,\dots,E_6\subset S$, we obtain a birational morphism $S\to \p^2$ that contracts the six curves onto six points $p_1,\dots,p_6$. We denote by $L\in \Pic(S)$ the pull-back of a general line. 
If $C$ is equal to one of the six lines $E_1,\dots,E_6$, then $(g,d)=(0,1)$. 
We can thus assume that the curve $C$ is the strict transform of a curve of $\p^2$ of degree $k$ having multiplicity $m_1,\dots,m_6$ at the points $p_1,\dots,p_6$. 
Hence, $C\sim kL-\sum a_i E_i$. We take then a set of six lines such that the integer $k$ is minimal and order the lines such that $m_1\ge \dots \ge m_6$. This implies that $k\le m_1+m_2+m_3$. 
Indeed, otherwise the contraction of the six curves $L-E_1-E_2,L-E_1-E_3,L-E_2-E_3,E_4,E_5,E_6$ would yield a curve of $\p^2$ of degree $2k-m_1-m_2-m_3<k$.

As there is no $3$-secant line to $C$, we have $a_i=C\cdot E_i\le 2$ for each $i$. Moreover, $C\cdot (L-E_5-E_6)\le 2$ and $C\cdot (2L-E_2-E_3-E_4-E_5-E_6)\le 2$ and in particular $k\le 6$. These conditions, together with $k\ge m_1+m_2+m_3$ yield the following five solutions
\begin{center}
\begin{tabular}{CCCC}
\toprule
g & d& k & (m_1,\cdots, m_6) \\
\midrule
0&2& 1 & (1,0,0,0,0,0) \\
0&3& 1 & (0,0,0,0,0,0)\\
1 & 3 &3 & (1,1,1,1,1,1) \\
1&4& 3 & (1,1,1,1,1,0) \\
4&6& 6 & (2,2,2,2,2,2) \\
\bottomrule
\end{tabular}
\end{center}
These, together with the lines $E_i$, correspond to the elements of $\Lp$. As explained in Lemma~\ref{Lemm:Smalldegreecurves}, each of these exists on any smooth cubic surface.
\end{proof}

\begin{prop}\label{Prop:Always3secant}
Let $Y\subset \p^4$ be a smooth cubic hypersurface and let $C\subset Y$ be a smooth curve of genus~$g$ and degree~$d$, such that one of the following holds:
\begin{enumerate}
\item \label{case:secant1} $(g,d)\in \{(2,5),(3,6)\}$;
\item \label{case:secant2} $(g,d)\in \Lq$  and $C$ is contained in a hyperplane section.
\end{enumerate}
Then $C$ is contained in a unique hyperplane section $S\subset Y$, and there exists a $3$-secant line to $C$ inside $S$.
In particular, if $X\to Y$ is the blow-up of $C$, then $-K_X$ is big but not nef.
\end{prop}
\begin{proof}
In case~\ref{case:secant1}, the fact that $C$ is contained in a hyperplane section $S\subset Y$  follows from Lemma~\ref{Lemm:hypersurfaces}\ref{item:planelist}. 
In case~\ref{case:secant2}, it is given by hypothesis.
The hyperplane, and thus the surface $S$, is unique since $d> 3$.

It is known (since either $g + 3 \le d$ or $g + 9 \le 2d \le 22$, see \cite{Guf}) that the Hilbert scheme $\mathcal{H}^S_{g,d}$ parametrising the smooth curves of such genus~$g$ and degree~$d$ in $\p^3$ is irreducible. We denote by $W$ the projective space parametrising cubics of $\p^3$, and consider the closed subset $Z\subset \mathcal{H}^S_{g,d} \times W$ consisting of pairs $(D,S)$ where $D\subset S$. 
We then denote by $Z_3\subset Z$ the subset of pairs $(D,S)$ such that $D$ admits a $3$-secant line contained in $S$. 
The set $Z_3$ is closed in $Z$, and it remains to show that $Z_3=Z$. 

Denote by $U\subset Z$ the open set of pairs $(D,S)$ where $S$ is a smooth cubic.  Lemma~\ref{Lemm:smoothcubics} implies that $U\subset Z_3$. It remains then to see that $Z$ is irreducible. Since $W$ is projective, the projection $\mathcal{H}^S_{g,d} \times W\to \mathcal{H}^S_{g,d}$ is closed. The restriction to the closed subset $Z$ is then a closed surjective morphism $Z\to \mathcal{H}^S_{g,d}$. The image being irreducible and the fibres too (it is a linear subspace of $W$), the variety $Z$ is also irreducible.
\end{proof} 

\subsection{Summary of necessary conditions}\label{Sec:SummaryNec}

We can now summarise the results obtain in $\S\ref{Sec:Preliminaries}$--$\S\ref{Sec:hyperplane}$ and obtain a part of the proof of Theorem~\ref{Thm:WeakFano}$\ref{case:thm1}$.
\begin{prop}\label{Prop:NecessaryConditionsPropre}
Let $Y\subset \p^4$ be a smooth cubic threefold and let $C\subset Y$ be a smooth curve of type $(g,d)$ such that the blow-up $X\to Y$ of~$Y$ along $C$ is weak-Fano. Then, the following hold:
\begin{enumerate}
\item
If $C$ is contained in a hyperplane section, then $(g,d)\in \Lp$.
\item
If $C$ is not contained in a hyperplane section, then $(g,d)\in \Lq$.
\end{enumerate}
Moreover, there is no $3$-secant line to $C$ in $Y$.
\end{prop}
\begin{proof}
According to Proposition~\ref{Prop:12 possibilities}, the pair $(g,d)$ belongs to $\Lp\cup \Lq \cup \{(2,5),(3,6)\}$.
If $C$ is contained in a hyperplane section, Proposition~\ref{Prop:Always3secant} shows that $(g,d)\notin \Lq\cup \{(2,5),(3,6)\}$.
If $C$ is not contained in a hyperplane section, $(g,d)\notin \Lp\cup \{(2,5),(3,6)\}$ by Lemma~\ref{Lemm:hypersurfaces}\ref{item:planelist}. Moreover, there is no $3$-secant line to $C$ in $Y$ since the strict transform of such a curve would intersect $-K_X$ negatively.\end{proof}

\section{Curves of type $(g,d)\in\Lp\cup \Lq$ and proof of Theorem~\ref{Thm:WeakFano}}\label{Sec:ProofofTheoremA}

In Section~\ref{Sec:NecessarySection}, we showed that curves that yield a weak-Fano threefold are of type $(g,d)\in\Lp\cup \Lq$, do not have any $3$-secant line and are contained in a hyperplane section if and only if $(g,d)\in \Lp$. In this section, we do the converse and show that each curve that satisfies these conditions yield a weak-Fano threefold. This will lead to the proof of Theorem~\ref{Thm:WeakFano}, done in $\S\ref{Sec:SummaryProof}$--$\ref{Sec:links}$.

\subsection{Fano case}\label{Sec:FanoCase}
First we treat the two elementary cases $(g,d)\in \{(0,1),(1,3)\}$, which yield Fano threefolds.

\begin{lemm}\label{Lemm:fano2}
Let $C \subset Y \subset \p^4$ be a smooth curve in a smooth cubic threefold, and consider $\pi\colon X \to Y$ the blow-up of $C$.
Suppose that
\begin{enumerate}
\item the linear system $\lvert-K_X\rvert$ is non empty and has no fixed component;
\item $(-K_X)^3 > 0$;
\item $-K_X\cdot \Gamma > 0$ for any curve $\Gamma \subset X$. 
\end{enumerate}
Then $X$ is a Fano threefold. 
\end{lemm}

\begin{proof}
The Nakai-Moishezon criterion (see \cite{Laz1}) asserts that $-K_X$ is ample if and only if $(-K_X)^3 > 0$, $(-K_X)^2\cdot S > 0$ and $-K_X\cdot \Gamma > 0$ for any surface $S \subset X$ and curve $\Gamma \subset X$.
But since we assume $\lvert-K_X\rvert$ without fixed component, $-K_X\cdot S$ is equivalent to an effective $1$-cycle, hence the condition on curves is enough.
\end{proof}

\begin{prop} \label{Prop:Fano}
Let $C \subset Y \subset \p^4$ be a smooth curve of genus~$g$ and degree~$d$ in a smooth cubic threefold, and let $\pi\colon X \to Y$ the blow-up of $C$.
If $(g,d) = (0,1)$ or $(1,3)$, then $X$ is a Fano threefold.
\end{prop}

\begin{proof}
In both cases, $X$ is a threefold with Picard number~$2$, which admits a fibration $\sigma\colon X \to Z$.
Precisely, when $C$ is a line, consider the family of 2-dimensional plane of $\p^4$ containing $C$. 
Each such plane defines a residual conic on $Y$, and this family is parametrised by $\p^2$.
So here $Z = \p^2$, and $\sigma$ is a conic bundle.
On the other hand, if $C$ is a plane elliptic curve, $C$ is the complete intersection of two hyperplanes sections on $Y$. 
The family of hyperplane sections containing $C$ (or equivalently, containing the $2$-dimensional plane containing $C$) is parametrised by $Z = \p^1$, and here $\sigma$ is a Del Pezzo fibration of degree $3$.

The class of a curve contracted by $\pi$ or $\sigma$ is extremal in the cone $\NE(X)$ of effective curves on $X$, and since this cone is 2-dimensional, such curves generate $\NE(X)$.
Moreover any such curve $\Gamma$ (a fibre of the exceptional divisor, a residual conic or a curve in a hyperlane section) satisfies~$-K_X\cdot\Gamma > 0$.
The linear system $\lvert -K_X \rvert$ is base-point free in both cases, and by Lemma~\ref{Lemm:K^3} we have $(-K_X)^3 > 0$.
So we conclude by Lemma~\ref{Lemm:fano2}.
\end{proof}

\subsection{Existence of smooth hyperquadric sections}\label{Sec:smooth_sections}
In this section, we take a smooth curve $C$ of type $(g,d)$, either with $(g,d) \in \Lp$, or with  $(g,d) \in \Lq$ and the extra assumptions that $C$ is not in a hyperplane and does not have a $3$-secant line (in the case $(0,5)$). 
Then, we shall show that $C$ is always contained in a smooth  hyperquadric section,  that the blow-up of $C$ is weak-Fano, and has an anti-canonical linear system without base-point.\\

First we treat two special cases, which correspond to the complete intersections of~$Y$ with two hyperplanes or with a hyperplane and a quadric.

\begin{prop} \label{Prop:cases 13 and 46}
Let $C \subset Y$ be a smooth curve of genus~$g$ and degree~$d$ in a smooth cubic threefold, with $(g,d)\in \{(1,3),(4,6)\}.$
Denoting by $X\to Y$ the blow-up of~$Y$ along $C$, the following assertions hold:
\begin{enumerate}
\item \label{case1:cases13and46}
The linear system of hyperquadric sections of~$Y$ through $C$ has no base-point outside $C$ and has a general member which is smooth.
\item \label{case2:cases13and46}
The linear system $\lvert -K_X\rvert$ has no base-point and has a general member which is smooth.
\item \label{case3:cases13and46}
The threefold $X$ is weak-Fano.
\end{enumerate}
\end{prop}
 
\begin{proof}
By Lemma~\ref{Lemm:hypersurfaces}, the curve $C$ is contained in a surface $S=H\cap W$, where $H\subset \p^4$ is a hyperplane and $W\subset \p^4$ is a hyperplane or a quadric, in case $(g,d)=(1,3)$ or $(g,d)=(4,6)$ respectively.

Since $S\cap Y$ has degree~$d$, we have $C=S\cap Y$. This implies in particular that any quadric of $\p^4$ passing through $C$ also contains $S$, but also, since $C$ is smooth, that the intersection $H\cap W \cap Y$ is transversal at each point of $C$. 
In particular, $W$ is smooth along $C$.

Let $\Lambda$ be the linear system of quadrics of $\p^4$ passing through $C$. We now show that the base-locus of $\Lambda$ is equal to $S$. If $(g,d)=(1,3)$ this is because $\Lambda$ contains all elements of the form $H'+H''$, where $H',H''$ are two hyperplanes, one of them containing the plane $S$. If $(g,d)=(4,6)$, this is because it contains $W$ and all elements of the form $H+H'$, where $H'$ is any hyperplane.

In particular, the linear system $\Lambda_Y=\Lambda|_Y$ of hyperquadric sections of~$Y$ containing $C$ has no base-point outside $C=S\cap Y$. 
By Bertini's Theorem, this implies that a general member of $\Lambda_Y$ is smooth outside of $C$, and then that a general member of $\lvert -K_X\rvert$ is smooth outside of the exceptional divisor. It remains to show that a general member of $\Lambda_Y$ is in fact smooth at every point of $C$, and that $\lvert -K_X\rvert$ has no base-point. This will yield \ref{case2:cases13and46} and that $-K_X$ is nef, hence big since $(-K_X)^3=22-4d+2g>0$, which implies \ref{case3:cases13and46}.

Suppose that $(g,d)=(1,3)$, and take coordinates $[v:w:x:y:z]$ on $\p^4$ such that $H$ and $W$ are given by $v=0$ and $w=0$ respectively. We can assume that $[0:0:1:0:0]\notin C$, and see that the quadric $Q\subset \p^4$ given by $vz+wy=0$ has tangent $bw+cv=0$ at a point $[0:0:a:b:c]$. In particular $Q\cap Y$ is smooth at any point of $C$ since $Y$ is transversal to the plane $S$ given by $v=w=0$. This provides a member of $\Lambda_Y$ which is smooth along $C$. Moreover, changing the equation of $Q$ (by taking $vy+\mu wz=0$, $\mu\in \C$ for instance), the tangent of the elements of $\Lambda$ at a point $p\in C$ cover all hyperplanes of $T_p\p^4$ containing $T_p S$. Again, as $Y$ is transversal to $C$, this implies that $\lvert -K_X\rvert$ has no base-point on the fibre above $p$.

The remaining case is $(g,d)=(4,6)$. In this case, $W$ is a quadric transversal to $Y$ along $C$, so $W|_Y$ yields a member of $\Lambda_Y$ which is smooth along $C$. Fixing a point $p\in C$, we can take elements of the form $H+H'$, where $H'$ is a hyperplane away from $p$, and obtain an element of $\Lambda_Y$ that is smooth at~$p$, with tangent corresponding to $H$. Since $C=H\cap W\cap Y$ is smooth at~$p$, this implies that the system $\lvert-K_X\rvert$ has no base-point above $p$.\end{proof}

For the other cases, our strategy consists of showing first that the curve is contained in the smooth intersection of two quadrics of $\p^4$, that gives a Del Pezzo surface of degree $4$ (Proposition~\ref{Prop:DelPezzo4}). This will be used to show the existence of smooth hyperquadric sections in Proposition~\ref{Prop:smoothgeneral}.

In order to show the existence of the Del Pezzo surface, we first need two lemmas about the possible embeddings of curves from the list $\Lq$ in $\p^4$.

\begin{lemm} \label{Lemm:DelPezzo4easy}
Let $C, \tilde C \subset \p^4$ be smooth curves of genus~$g$ and degree~$d$ and let $H$ be a general hyperplane section of $C$. 
We assume:
\begin{enumerate}
\item Either $(g,d)=(1,4)$, or $(g,d)\in\{\, (1,5),\, (2,6)\,\}$ and $C$ is not contained in a hyperplane.
\item There exists an isomorphism $\phi\colon C \to \tilde C$ that sends $H$ to a hyperplane section of $\tilde C$.
\end{enumerate}
Then there exists an isomorphism of $\p^4$ that extends $\phi$.
\end{lemm}

\begin{proof}
By Riemann-Roch's formula we have $\lvert H\rvert \simeq \p^{d-g}$ on $C$.
In particular in cases $(g,d) = (1,5)$ or $(2,6)$ we have a one-to-one correspondence between elements of $\lvert H\rvert$ and hyperplanes in $\p^4$, since  we assume that $C$ is not contained in a hyperplane.
Thus the isomorphism $H' \in \lvert H\rvert \mapsto \phi(H') \in \lvert \phi(H)\rvert $ corresponds to the expected automorphism of $\p^4$.

In the case $(g,d) = (1,4)$ this is the same argument, working with the three-dimensional system of planes in the unique $\p^3$ containing $C$.
\end{proof}

The case of a curve of type $(0,5)$ needs more assumptions, since in this case the restriction of the linear system of hyperplanes is not complete. 
We will also need Lemma~\ref{Lemm:HRS9.3}, which says that each such curve admits a unique $3$-secant in $\p^4$.

\begin{lemm}\label{Lemm:DelPezzo4hard}
Let $C, \tilde C \subset \p^4$ be smooth rational curves of degree $5$, and let $\phi \colon C \to \tilde C$ be an isomorphism.
Let $M, \tilde M$ be the $3$-secant lines to $C$ and $\tilde C$ respectively, and let $L, \tilde L$ be $2$-secant lines disjoint from $M$ and $\tilde M$. 
Denote by $\pi, \tilde \pi\colon \p^4 \dasharrow \p^2$ the projections from $L$ and $\tilde L$.
Assume
\begin{enumerate}
\item $\pi|_{C} = \tilde \pi \circ \phi$;
\item $\phi (M \cap C) = \tilde M \cap \tilde C$ and $\phi (L \cap C) = \tilde L \cap \tilde C$.
\end{enumerate}
Then there exists an automorphism of $\p^4$ that extends $\phi$.
\end{lemm}

\begin{proof}
Up to composition by an automorphism of $\p^4$, we can assume that 
\begin{align*}
M = \tilde M &= \{[0:0:0:x_3:x_4]\}; \\
L = \tilde L &= \{[x_0:x_1:0:0:0]\}.
\end{align*}
and $M \cap C = \tilde M \cap \tilde C$, $L \cap C = \tilde L \cap \tilde C$.
Moreover, up to choosing coordinates of $\p^2$, both maps $\pi,\tilde \pi$ can be chosen to be 
$$[x_0:x_1:x_2:x_3:x_4]\dashmapsto [x_2:x_3:x_4].$$

Let $\nu, \mu\colon \p^1 \to \p^4$ be parametrisations of $C$, $\tilde C$ such that $\mu = \phi \circ \nu$.
Let $f_2$ (resp.~$f_3$)  be a homogeneous polynomial in $\C[u,v]$ of degree 2 (resp.~3), the roots of which correspond to the points of $\p^1$ sent to $C \cap L$ (resp.~$M \cap C$).

Now we have
\begin{align*}
\nu\colon [u:v]\in \p^1 &\mapsto [f_3a_2:f_3b_2:f_3f_2:g_3f_2:h_3f_2],\\
\mu\colon [u:v]\in \p^1 &\mapsto [f_3a'_2:f_3b'_2:f_3f_2:g'_3f_2:h'_3f_2],
\end{align*}
for some homogeneous polynomials $a_2, a_2', b_2,b'_2, g_3, g'_3, h_3, h'_3$, where the subscript corresponds to the degree.
Since we assume $\pi = \tilde \pi \circ \phi$, we have $g_3 = g'_3$ and $h_3 = h'_3$.

The projections of $C$ and $\tilde{C}$ from the $3$-secant correspond to the projection on the first $3$ coordinates, and yields a smooth conic. Choosing an automorphism of $\p^2$ on the first three coordinates, that fix the third one, we obtain that the two parametrisations of smooth conics
$$[u:v] \mapsto [a_2:b_2:f_2] \text{ and } [u:v] \mapsto [a'_2:b'_2:f_2] $$
are the same.
\end{proof}

\begin{prop}\label{Prop:DelPezzo4}
Let $C\subset \p^4$ be a smooth curve of genus~$g$ and degree~$d$, such that one of the following holds:
\begin{enumerate}
\item
$(g,d)\in\{\, (0,1),\, (0,2),\, (0,3),\,  (1,4)\, \}=\Lp\smallsetminus\{\,(1,3),(4,6)\,\}.$
\item
$(g,d)\in\{\, (0,4),\, (0,5),\,  (1,5),\,(2,6)\, \}=\Lq$ and $C$ is not contained in a hyperplane.\end{enumerate}
Then the curve $C$ is contained in a Del Pezzo surface $S\subset \p^4$ of degree~$4$, smooth intersection of two quadrics in $\p^4$. Moreover, we can choose a birational morphism $S\to \p^2$, blow-up of five points $p_1,\dots,p_5$, such that the curve $C$ is the strict transform of a curve of degree $k$  with multiplicity $m_i$ at $p_i$, according to Table~$\ref{tab:DelPezzo4}$.
\end{prop}
\begin{table}[h]
\begin{center}
\begin{tabular}{CCCCC}
\toprule 
g & d & k & (m_1, \dots, m_5)  \\ 
\midrule
0 & 1 & 2 & (1,1,1,1,1) \\
0 & 2 & 2 & (1,1,1,1,0) \\
0 & 3 & 2 & (1,1,1,0,0) \\
0 & 4 & 2 & (1,1,0,0,0) \\
0 & 5 & 3 & (2,1,1,0,0) \\
1 & 4 & 3 & (1,1,1,1,1) \\
1 & 5 & 3 & (1,1,1,1,0) \\
2 & 6 & 4 & (2,1,1,1,1) \\
\bottomrule \\
\end{tabular}
\end{center}
\caption{}
\label{tab:DelPezzo4}
\end{table}
\begin{proof}
Recall the following classical fact. If $\pi \colon S\to \p^2$ is the blow-up of five points, such that no three are collinear, then $S$ is a Del Pezzo surface of degree $4$ and the anticanonical morphism yields a closed embedding $S\to \p^4$, which is the smooth intersection of two smooth quadric hypersurfaces. Moreover, every smooth intersection of two quadrics surface is obtained by this way.

Denoting by $E_1,\dots,E_5\subset S$ the exceptional divisors contracted by $\pi$ and by $L$ the pull-back of a general line, the  divisors $kL-\sum m_i E_i$ given in Table~\ref{tab:DelPezzo4} yield the desired pairs $(g,d)$ in $\p^4$. It remains to see that $C$ is given by one of these divisors, in some Del Pezzo surface of degree~$4$. 

If $(g,d)\in \{(0,1),(0,2),(0,3),(0,4)\}$, this is because the curve is unique up to automorphisms of $\p^4$ (by hypothesis, $C$ is not contained in a hyperplane in the case $(0,4)$).

In the other cases, we use the fact that $S\to \p^2$ corresponds to the projection of $\p^4$ from the line corresponding to the strict transform of the conic through $p_1,\dots,p_5$. 
Our strategy is then to choose a general line $L\subset \p^4$ having the correct intersection with $C$, project from it and choose some points in $\p^2$ according to the numerology in Table~\ref{tab:DelPezzo4}.

If $(g,d)\in (1,5)$, we project from a $2$-secant line $L$ of $C$ and obtain a rational map $\pi\colon \p^4\dasharrow \p^2$. We claim that the restriction of $\pi$ to $C$ is birational. Indeed, the curve is not in a hyperplane so the image $C'$ is a curve of degree $\ge 2$; the degree of $C'$ times the degree of the map being $3$, we find that $C'$ is a cubic and the restriction is birational. The genus of $C$ being $1$, the curve $C'$ is smooth and $\pi$ restricts to an isomorphism $C\to C'$. 
Let $q_1,q_2\in C'$ be the image of the two points of $L\cap C$,  let $D'\subset \p^2$ be a general conic through $q_1,q_2$ and denote by $p_1,\dots,p_4$ the points such that $C' \cap D' =\{q_1,q_2,p_1,\dots,p_4\}$.
We then choose a general point $p_5$ of $D'$, denote by $\kappa\colon S\to \p^2$ the blow-up of $p_1,\dots,p_5$ and see the Del Pezzo surface $S$ as a surface of degree $4$ in $\p^4$, via the anticanonical morphism. 
Writing as before $E_i=\kappa^{-1}(p_i)$, for $i=1,\dots,5$ and denoting by $\tilde L\in \Pic{S}$ the divisor corresponding to a line of $\p^2$, the strict transform $\tilde{C}\subset S$ of the curve $C'$ is linearly equivalent to $3\tilde L-\sum_{i=1}^4 E_i$ as desired. 

Now observe that $\tilde L+\tilde{D} = \tilde L+(2\tilde L-\sum_{i=1}^5 E_i)= 3\tilde L-\sum_{i=1}^5 E_i$ is an hyperplane section of $S$.
Restricting to $\tilde C$ we obtain that $\tilde{L}|_{\tilde C}+\tilde{D}|_{\tilde C}=\tilde{L}|_{\tilde C}+\kappa^{-1}(q_1)+\kappa^{-1}(q_2)$ is a hyperplane section.
This corresponds via $\kappa^{-1}\pi$ to the restriction of a hyperplane through $L$ to $C\subset \p^4$.
So we can apply Lemma~\ref{Lemm:DelPezzo4easy} and obtain that $C$ and $\tilde C$ are projectively equivalent; in particular the curve $C$ is also contained in a smooth Del Pezzo surface of degree 4.

The same kind of argument works with $(g,d)\in \{(1,4),(2,6)\}$, we briefly explain the slight differences. For $(2,6)$, we also choose a $2$-secant line $L$ and get a curve $C'\subset \p^2$ of degree $4$, with a double point $p_1$, if the restriction to $C$ is birational. If it is not birational, then it is a double covering over a conic of $\p^2$, which implies that the restriction of the linear system to $C$ is equivalent to $|2K_C|$, since $K_C$ is the unique divisor of degree $2$ that gives a double covering. Hence, $L|_C+2K_C$ is linearly equivalent to the system of hyperplanes. Choosing  a general $2$-secant line avoids this situation, since the divisors of degree $2$ are not all equivalent. We then choose a general conic $D'$ through $p_1,q_1,q_2$, and the intersection with $C'$ yields then four other points $p_2,\dots,p_5$. On the blow-up we again get $\tilde{D}\cap {\tilde{\Gamma}}=\{\kappa^{-1}(q_1),\kappa^{-1}(q_2)\}$, so the same argument works. 
For $(1,4)$, we choose a $1$-secant line $L$ which is not contained in the hyperplane containing $C$ and get a smooth cubic curve $C'\subset \p^2$. We then choose a general conic $D'$ through the point $q_1$ corresponding to $C\cap L$, and get five other points.

Now we consider the remaining case $(g,d) = (0,5)$.
Recall from Lemma~\ref{Lemm:HRS9.3} that such a curve $C$ admits a unique 3-secant line $M$ in $\p^4$.
We choose a general $2$-secant line $L\subset \p^4$ of $C$, and consider as before the projection $\pi\colon \p^4\dasharrow \p^2$ by $L$, which sends $C$ onto a cubic curve $C'$, with a double point at $p_1\in \p^2$. 
Moreover, $\pi$ sends $M$ onto a line $M'\subset \p^2$. 
Denote by $q_1,q_2\in C'$ the image of $L\cap C$. 
Let $D'\subset \p^2$ be a general conic through $p_1,q_1,q_2$, whose intersection with $C'$ gives two other points $p_2,p_3\in \p^2$ such that $C'|_{D'}=2p_1+q_1+q_2+p_2+p_3$. 
We then denote by $p_4,p_5$ the two points of $D'\cap M'$. 
As before, we denote by $\kappa\colon S\to \p^2$ the blow-up of $p_1,\dots,p_5$, and see the Del Pezzo surface in $\p^4$. 
The strict transforms of $C',M',D'$ are then equivalent to $k\tilde L-\sum m_i E_i$, according to the multiplicities of Table~\ref{tab:05}.
\begin{table}[h]
\begin{center}
\begin{tabular}{CCCCCC}
\toprule 
&g & d & k & (m_1, \dots, m_5)  \\ 
\midrule
C'\cup M'&2 & 6 & 4 & (2,1,1,1,1) \\
C'&0 & 5 & 3 & (2,1,1,0,0) \\
M'&0 & 1 & 1 & (0,0,0,1,1) \\
D'&0 & 1 & 2 & (1,1,1,1,1) \\
\bottomrule\\
\end{tabular}
\end{center}
\caption{}
\label{tab:05}
\end{table}

Then one can apply Lemma~\ref{Lemm:DelPezzo4hard}.
\end{proof}

\begin{prop} \label{Prop:smoothgeneral}
Let $Y\subset \p^4$ be a smooth cubic, and let $C \subset Y$ be a smooth curve of genus~$g$ and degree~$d$, such that one of the following holds:
\begin{enumerate}[$(i)$]
\item
$(g,d)\in \Lp;$
\item
$(g,d)\in \Lq\smallsetminus \{(0,5)\}$ and $C$ is not contained in a hyperplane;
\item
$(g,d)=(0,5)$, $C$ is not contained in a hyperplane and the unique $3$-secant 
line of $C$  in $\p^4$ is not contained in 
$Y$.
\end{enumerate}Denoting by $X\to Y$ the blow-up of~$Y$ along $C$, the following assertions hold:
\begin{enumerate}
\item
The linear system of hyperquadric sections of~$Y$ through $C$ has no base-point outside $C$ and has a general member which is smooth.
\item
The linear system $\lvert -K_X\rvert$ has no base-point and has a general member which is smooth.
\item
The threefold $X$ is weak-Fano.
\end{enumerate}
\end{prop}

\begin{proof}
We can assume $(g,d) \not\in \{(1,3),(4,6)\}$, otherwise Proposition~\ref{Prop:cases 13 and 46} applies directly.

Let $\Lambda$ be the linear system of quadrics of $\p^4$ passing through $C$. We denote by $\pi\colon \hat{\p}^4\to \p^4$ the blow-up of $C$, and denote by $\hat{\Lambda}$ the linear system of hypersurface of $\hat{\p}^4$ that correspond to strict transforms of elements of $\Lambda$. Let us show that $\hat{\Lambda}$ is without base-point, except in the case $(g,d)=(0,5)$, where its base-locus is the strict transform $\tilde{L}$ of the unique $3$-secant line $L\subset \p^4$ of $C$.

To show this, we use the fact that $C$ is contained in a Del Pezzo surface $S$ of degree $4$, intersection of two smooth quadrics $Q_1$ and $Q_2$ (Proposition~\ref{Prop:DelPezzo4}). We denote by $\Lambda_S=\Lambda|_S$ the linear system $\Lambda_S$ of hyperquadric sections of $S$ containing $C$, and observe that the residual system $R = \Lambda_S - C$ (respectively $R = \Lambda_S - C-D$ in the case $(0,5)$) is base-point free. This can be checked by looking at Table~\ref{tab:4} (using the notation of Proposition~\ref{Prop:DelPezzo4}).
\begin{table}[h]
\begin{center}
\begin{tabular}{CCCCCC}
\toprule 
g & d & k & (m_1, \dots, m_5) & R \\ 
\midrule
0 & 1 & 2 & (1,1,1,1,1) & 4 (1,1,1,1,1)\\
0 & 2 & 2 & (1,1,1,1,0) & 4 (1,1,1,1,2)\\
0 & 3 & 2 & (1,1,1,0,0) & 4 (1,1,1,2,2)\\
0 & 4 & 2 & (1,1,0,0,0) & 4 (1,1,2,2,2)\\
0 & 5 & 3 & (2,1,1,0,0) & 2 (0,1,1,1,1)\\
1 & 4 & 3 & (1,1,1,1,1) & 3 (1,1,1,1,1)\\
1 & 5 & 3 & (1,1,1,1,0) & 3 (1,1,1,1,2)\\
2 & 6 & 4 & (2,1,1,1,1) & 2 (0,1,1,1,1)\\
\bottomrule
\end{tabular}
\end{center}
\caption{} \label{tab:4}
\end{table}
Indeed, conics of $\p^2$ through four points correspond to a pencil of conics of $\p^4$, which has no base-point. All the above are sum of one such system with other base-point-free systems. 

This gives the desired result on the system $\hat{\Lambda}$. Indeed, the strict transform of $Q_1$ and $Q_2$ on $\hat{\p}^4$ are two elements of $\hat{\Lambda}$, whose intersection is a surface $\tilde{S}$ isomorphic to $S$, and the restriction of $\hat{\Lambda}$ to $\tilde{S}$ corresponds to $\Lambda_S-C$.

The blow-up $X\to Y$ of~$Y$ along $C$ is naturally embedded into $\hat{\p}^4$, and we have $\lvert-K_X\rvert=\hat{\Lambda}|_X$. 
In particular, the linear system $\lvert -K_X\rvert$ is base-point-free: this is clear except in the case $(g,d)=(0,5)$, where we have to observe that $\tilde{L}$ and $X$ are disjoint. 
Indeed, denoting by $E_{\hat{\p}^4}\subset \hat{\p}^4$ the exceptional divisor of $\pi$, we have $\tilde{L}\cdot E_{\hat{\p}^4}\ge 3$ and $X\simeq \pi^*(Y)-E_{\p^4}$, so $\tilde{L}\cdot X=0$. 
This implies that the base-locus of $\Lambda_Y=\Lambda|_Y$ is equal to $C$, but also that $-K_X$ is nef. By Lemma~\ref{Lemm:K^3} we have $(-K_X)^3 =22-4d+2g>0$. Thus $X$ is weak-Fano.

By Bertini's Theorem, a general member of $\Lambda_Y$ is smooth outside of $C$. It remains to show that a general member of $\Lambda_Y$ is smooth at every point of $C$, this will imply that a general member of $\lvert -K_X\rvert$ is smooth.

We assume that $p \in C$ is a singular point of a hyperquadric section $Q \cap Y$. 
This is equivalent to the equality $T_p Q = T_p Y$ as 3-dimensional vector spaces of $T_p \p^4$. 
Since $Q_1 \cap Q_2$ is a smooth surface, we have $T_p Q_1 \neq T_p Q_2$ for any $p \in C$, and by the same argument the same is true for any two members of the pencil generated by $Q_1$ and $Q_2$. 
In particular we can have $T_p Q = T_p Y$ only for a unique member $Q$ of the pencil generated by $Q_1$ and $Q_2$: this shows that the hypothetical singularity of the system $\Lambda_Y$ is mobile along $C$.

In particular, the restriction of a general member of $\lvert-K_X\rvert$ to the exceptional divisor $E$ has the form $s + \sum f_i$, where $s$ is a section of $E\to C$ and each $f_i$ is a fibre above a singular point.
In particular this general member is reducible hence singular.
Since the $f_i$ are mobile, by Bertini's Theorem this is possible only if $s$ is a fixed 
part of the linear system, which contradicts the fact that $\lvert -K_X\rvert$ is base-point free.
\end{proof}

Let us summarise the results obtained so far:

\begin{coro}\label{Coro:WeakFano}
Let $Y\subset \p^4$ be a smooth cubic hypersurface, let $C\subset Y$ be a smooth curve of type $(g,d) \in  \Lp$ and let $X\to Y$ be the blow-up of~$C$.
\begin{enumerate}
\item \label{case:weakFano1} If $(g,d)\in \{\,(0,1),\,(1,3)\,\}$, then $X$ is Fano;
\item \label{case:weakFano2} If $(g,d)\in \{\,(0,2),\,(0,3),\,(1,4),\,(4,6)\,\}=\Lp\smallsetminus \{\,(0,1),\,(1,3)\,\}$, then $X$ is weak-Fano.
\item \label{case:weakFano3} If $(g,d)\in \{(0,4),\, (1,5),\, (2,6)\}=\Lq\smallsetminus \{(0,5)\}$ and $C$ is not contained in a hyperplane then $X$ is weak-Fano.
\item
If $(g,d)=(0,5)$, $C$ is not contained in a hyperplane and the unique $3$-secant 
line of $C$  in $\p^4$ is not contained in 
$Y$, then $X$ is weak-Fano.
\end{enumerate}
Moreover, in all these cases, there exists a smooth hyperquadric section of~$Y$ that contains $C$, and $\lvert -K_X\rvert$ has no base-point.
\end{coro}

\begin{proof}
Case~\ref{case:weakFano1} was proved in Proposition~\ref{Prop:Fano}. Case~\ref{case:weakFano2} is given by Proposition~\ref{Prop:cases 13 and 46}. 
The two other cases are given by Proposition~\ref{Prop:smoothgeneral}. The existence of the smooth hyperquadric section is also provided by Propositions~\ref{Prop:cases 13 and 46} and~\ref{Prop:smoothgeneral}.
\end{proof}

\begin{rem} \label{Rem:cas05}Let $Y\subset \p^4$ be a smooth cubic and let $C\subset Y$ be a smooth curve of type $(g,d)\in \Lq.$
\begin{enumerate}[wide]
\item
We have seen in Proposition~\ref{Prop:Always3secant} that the existence of a hyperplane section containing $C$ yields the existence of a line $L\subset Y$ that is  $3$-secant  to~$C$.
\item
Conversely, the existence of the $3$-secant yields the existence of the hyperplane section, if $(g,d)\not=(0,5)$ (Corollary~\ref{Coro:WeakFano}).
\item
Since by Lemma~\ref{Lemm:HRS9.3} every non-degenerate curve of type $(0,5)$ in $\p^4$ admits a $3$-secant line, this does not generalise to the case $(g,d)=(0,5)$.
\end{enumerate}\end{rem}

\subsection{Existence of curves in hyperquadric sections}\label{Sec:hyperquadric}

In order to prove Theorem~\ref{Thm:WeakFano}, it remains to show the existence of the four cases $(g,d)\in \Lq$, when $C$ lies in a hyperquadric section but not in a hyperplane section.
The aim of this section is to produce examples of such smooth curves that give rise to a nef anticanonical divisor after blow-up.
The point here is to show that the open set in Theorem~\ref{Thm:WeakFano}\ref{case:thm2}\ref{case:thmii} is not empty.
We shall produce these curves  by considering singular rational hyperquadric surfaces.\\ 

The following result is classical, we recall the proof for sake of completeness.

\begin{lemm}\label{Lemm:6lines}
Let $Y\subset \p^4$ be a smooth cubic threefold. There exists a dense open subset $U\subset Y$ such that for each point $p\in Y$, the following hold:
\begin{enumerate}[wide]
\item The tangent hyperplane section at $p$ is smooth outside of $p$ and has a simple double point at $p$;
\item There are exactly six lines of~$Y$ that pass through $p$.
\end{enumerate}
\end{lemm}

\begin{proof}
\begin{enumerate}[wide]
\item We consider the Gauss map $\kappa\colon Y\to (\p^4)^\vee$ that sends a point onto the corresponding tangent hyperplane. Since $Y$ is smooth, $\kappa$ is a regular morphism. Moreover, 
$\kappa$ is given by the first derivatives of the equation of~$Y$, and is then the restriction of a rational map $\p^4\to (\p^4)^\vee$ of degree $2$. The Picard group of~$Y$ is $\mathbb{Z}H$, where $H$ is a hyperplane section, and the linear system given by $\kappa$ consists then of a subsystem of $\lvert 2H\rvert$, so any member of this system is ample on $Y$. Hence, the restriction to any curve of~$Y$ is positive, which implies that no curve of~$Y$ is contracted by $\kappa$, so the closure of $\kappa(Y)$ is a hypersurface $Y^\vee\subset (\p^4)^\vee$. The Gauss map associated to $Y^\vee$ yields then a morphism from the smooth locus of~$Y^\vee$ to $Y$, which is the inverse of $\kappa$. Hence, $\kappa$ is a birational map from $Y$ to $Y^\vee$. This shows that the tangent hyperplane to a general point of~$Y$ is not tangent to any other point, and thus is smooth outside of $p$.

The subset $U$ of~$Y$ where $\kappa$ is locally injective is given by the points where the determinant of the Hessian matrix associated to the equation of~$Y$ is not zero. Since $\kappa$ is a birational map from $Y$ to $Y^\vee$, the open set $U$ is dense in~$Y$. The non-vanishing of the determinant of the Hessian at a point $p\in Y$ also corresponds to the fact that the tangent hyperplane section at $p$ has an ordinary double point (as one can check by working in coordinates). 

\item We take a general point $p\in Y$, and choose coordinates $[v:w:x:y:z]$ on $\p^4$ such that $p=[1:0:0:0]$ and that the tangent hyperplane section is $w=0$. The equation of~$Y$ is then 
$$v^2w+vF_2(w,x,y,z)+F_3(w,x,y,z)=0,$$
where $F_2,F_3$ are homogeneous polynomials of degree $2$ and $3$ respectively. The union of lines of~$Y$ through $p$ is the union of points where $v$, $F_2$ and $F_3$ vanish. It remains to see that the fact that the hyperplane section $v=0$ has an ordinary double point at $p$ and is smooth outside implies that the conics and cubics of $\p^2$ given by $F_2(0,x,y,z)$ and $F_3(0,x,y,z)$ intersect in exactly $6$ points.

The point $p$ being ordinary double point of the quadric $vF_2(0,x,y,z)+F_3(0,x,y,z)=0$, the polynomial $F_2(0,x,y,z)$ is the equation of a smooth conic.
We assume for contradiction that $F_2(0,x,y,z)$ and $F_3(0,x,y,z)$ do not intersect in exactly $6$ points. 
Up to change of coordinates, we can thus assume that $F_2(0,x,y,z)=x^2-yz$ and that $F_3(0,uv,u^2,v^2)$ is divisible by $u^2$. 
This implies that the cubic surface corresponding to the tangent hyperplane section has equation
$$v(x^2-yz)+\lambda yz^2+zR_2(x,y)+R_3(x,y),$$ where $\lambda\in \C$ and $R_2,R_3$ are homogeneous of degree $2$ and $3$ respectively. 
But then the point $[v:x:y:z]=[\lambda:0:0:1]$ would be singular. \qedhere
\end{enumerate}
\end{proof}

\begin{lemm}\label{LemmSetup}
Let $Y\subset \p^4$ be a smooth cubic threefold, let $p\in Y$ be a general point and let $S\subset Y$ be a general hyperplane section  that does not contain~$p$.
We define a rational map $\varphi\colon S\dasharrow Y$ by sending a general point~$q$ onto the point of~$Y$ that is collinear with $p$ and $q$. 
Then, the following hold:
\begin{enumerate}
\item
The closure of $\varphi(S)$ is a rational hyperquadric section $Q\subset Y$ singular at $p$.
\item
The restriction of $\varphi$ is a birational map $S\dasharrow Q$ that decompose as $\eta\sigma^{-1}$, where $\sigma\colon W\to S$ is the blow-up of $6$ points $p_1,\dots,p_6$ on a smooth hyperplane section $\Gamma_0$ of $S$ and $\eta\colon W\to Q$ is the contraction of the strict transform $\Gamma\subset W$  of $\Gamma_0$ on the point $p$.
\item Denoting by $E_1,\dots,E_{6}\subset W$ the exceptional curves of $W$ contracted by $\sigma$ onto $p_1,\dots,p_{6}$ and $H_S\in \Pic{W}$ the pull-back of a hyperplane section of $S$ $($which is equivalent to $-K_S)$, the curve $\Gamma\subset W$ is then linearly equivalent to $H_S-\sum E_i\sim-K_W$ and a hyperplane section of $Q$ corresponds to $H_Y=2H_S-\sum E_i\in \Pic{W}$.
\item None of the $27$ lines of $S$ passes through one of the points $p_1,\dots,p_6$.
\end{enumerate}
\end{lemm}

\begin{proof}
As before, we choose coordinates such that $p=[1:0:\dots:0]$ and such that the tangent hyperplane section is $w=0$.
The equation of~$Y$ is then
$$v^2w+vF_2(w,x,y,z)+F_3(w,x,y,z)=0.$$
Moreover we can assume that $S$ is given by $v=0$, which implies that it corresponds to the smooth cubic surface of $\p^3$ given by $F_3(w,x,y,z)=0$. By definition, the map $\varphi\colon S\dasharrow Y$ is given by 
$$[w:x:y:z]\dashmapsto [-\tfrac{F_2}{w}:w:x:y:z]=[-F_2:w^2:wx:wy:wz],$$
and $Q\subset Y$ is the hyperquadric section given by $vw+F_2(w,x,y,z)=0$: This yields $(1)$. 

The map $\varphi\colon S \dasharrow Q$ is birational with inverse given by the projection $[v:w:x:y:z]\dashmapsto [w:x:y:z]$. 
The base-points of $\varphi$ are given by $w=0$ and $F_2=0$, $F_3=0$ in $\p^3$, and are thus the intersection of $S$ with the six distinct lines $\ell_1,\dots,\ell_6$ of~$Y$ passing through $p$ (see Lemma~\ref{Lemm:6lines}).
Since $F_2(0,x,y,z)$ is the equation of a smooth conic (see the proof of Lemma~\ref{Lemm:6lines}), the linear system given by $\varphi$ consists of all hyperquadric sections of $S$ passing through the six points $p_1,\dots,p_6$. Assertion $(2)$ follows, as well as $(3)$.

For $i=1,\dots,6$, there are again $6$ lines passing through any general point of $\ell_i$. 
Hence, the family of lines of~$Y$ that touch the $\ell_i$ has dimension $1$. 
This implies that a general hyperplane section of~$Y$ (here $S$) does not contain any line that touch the $\ell_i$, hence $(4)$.
\end{proof}

Finally we need the following fact:

\begin{lemm}Taking the notation of Lemma~$\ref{LemmSetup}$, the following holds:
 \label{Lemm:easypeasy}
\begin{enumerate}
\item
Let $C\subset W$ be an irreducible curve, its image $\eta(C)\subset Q\subset Y\subset\p^4$ is a smooth curve if and only if $C$ is smooth and $C\cdot \Gamma=1$ in $W$.
\item
There are exactly $6$ lines contained in $Q$, all passing through $q$ and corresponding to the image of $E_i$ for some $i=1,\dots,6$.
\end{enumerate}
\end{lemm}

\begin{proof}
\begin{enumerate}[wide]
\item Denote by $\hat{\eta}\colon X\to \p^4$ the blow-up of $q=[1:0:0:0]$. It follows from Lemma~$\ref{LemmSetup}$ that  the strict transform of $Q$ on $X$ is isomorphic to the smooth surface $W$, and that $\eta\colon W\to Q$ is the restriction of $\hat{\eta}$.

The curve $C\subset W$ is then the strict transform of $\hat{\eta}(C)\subset \p^4$. Denoting by $E\subset X$ the exceptional divisor, the curve $\eta(C)$ is smooth if and only if $C$ is smooth and $C\cdot E=1$ in $X$. Since $\Gamma$ is the intersection of $E$ and $W$, the intersection $C\cdot E$ on $X$ is equal to the intersection $C\cdot \Gamma$ in $W$.

\item Each $E_i$ is isomorphic to $\p^1$ on $W$, its intersection with $\Gamma$ and $H_Y$ is $1$, so its image in $Q$ is again isomorphic to $\p^1$, of degree $1$ and passing through $\eta(\Gamma)=(1:0:0:0)=q$.

Let $C\subset W$ be a curve distinct from the $E_i$, isomorphic to $\p^1$ and whose image by $\eta$ is a line. It is linearly equivalent to $\sigma^*(D)-\sum a_i E_i$, where $D$ is an effective divisor of the cubic surface $S$ and  $a_i\ge 0$ for $i=1,\dots,6$, and its intersection with $\Gamma\sim H_S-\sum E_i$ and $H_Y\sim 2H_S-\sum E_i$ is respectively  $\eps=H_S \cdot C-\sum a_i\in \{0,1\}$ and $1=\eps +H_S\cdot C$. Note that $H_S\cdot C$ is the degree of $\sigma(C)$ in $S\subset \p^3$, so the only possibility is that $C$ is the strict transform of a line of $S\subset \p^3$ passing through one of the $p_i$, impossible by Lemma~$\ref{LemmSetup}$. \qedhere
\end{enumerate}
\end{proof}

\begin{prop}\label{Prop:ExistenceWithoutLines}
Let $Q$ be a singular hyperquadric section of a smooth cubic $3$-fold $Y \in \p^4$ as given in Lemma~$\ref{LemmSetup}$.
If $(g,d) \in \Lq$, then $Q$ contains smooth curves of type $(g,d)$ without any $3$-secant lines.
\end{prop}

\begin{proof}
We think of $Q$ as $\p^2$ blown-up at 6 general points $q_1,\dots,q_6$ on a cubic curve $\Gamma_0$, producing exceptional divisors $F_1, \dots, F_6$, and then blown-up again at 6 points $p_1,\dots,p_6$ (all lying on one conic) on the strict transform of $\Gamma_0$, and finally blow-down of $\Gamma_0$.

We consider a curve of degree $k$ in $\p^2$ passing with multiplicities $m_i$ through the $q_i$ and multiplicities $n_i$ through the $p_i$, as given in Table~\ref{tab:without line}, and we denote by $C$ the transform of this curve on $Q$.
 
\begin{table}[h]
\begin{center}
\begin{tabular}{CCCCC}
\toprule 
g & d & k & (m_1, \dots, m_6) & (n_1, \dots, n_6)  \\ 
\midrule
0 & 4 & 1 & (0,0,0,0,0,0) & (1,1,0,0,0,0) \\
0 & 5 & 2 & (1,1,0,0,0,0) & (1,1,1,0,0,0) \\
1 & 5 & 3 & (1,1,1,1,1,0) & (1,1,1,0,0,0) \\
2 & 6 & 4 & (2,1,1,1,1,1) & (1,1,1,1,0,0) \\
\bottomrule\\
\end{tabular}
\end{center}
\caption{}\label{tab:without line}
\end{table}

We check that $3k - \sum m_i - \sum n_i = 1$ in all four cases, hence the intersection of the curve with $\Gamma_0$ is $1$ just before the blow-down, and the resulting curve $C$ is smooth on $Q$.
A hyperplane section on $Q$ is equivalent to
$$6H_{\p^2} - 2\sum F_i - \sum E_i.$$
The genus and degree are given by the formulas:
\begin{align*}
g &= \frac{(k-1)(k-2)}2 - \sum  \frac{m_i(m_i-1)}2 - \sum  \frac{n_i(n_i-1)}2, \\
d &= 6k - 2\sum m_i - \sum n_i. 
\end{align*}
By Lemma~\ref{Lemm:easypeasy}, the only possibility to have a 3-secant line would be to have $n_i \ge 3$ for some $i$, and this is not the case.
\end{proof}

\subsection{Summary of the results: the proof of Theorem~\ref{Thm:WeakFano}}
\label{Sec:SummaryProof}

First we prove a lemma that will provide the density in the case $(g,d) = (2,6)$.

\begin{lemm} \label{lem:26sing}
Let $Y \subset \p^4$ be a smooth cubic threefold, and let $V$ be the variety parametrising smooth curves of type $(2,6)$ that are contained in some hyperplane section of $Y$.
Then each component of $V$ has dimension at most~$11$.
\end{lemm}

\begin{proof}
Since hyperplanes sections of $Y$ are parametrised by $\p^4$, it is sufficient to prove that the dimension of smooth curves of type $(2,6)$ contained in a given hyperplane section $S\subset Y$ is at most $7$.

We view $S$ as a cubic surface in $\p^3$. Since it is a hyperplane section of a smooth cubic threefold, $S$ is irreducible and its singularities are isolated (as can be checked in local coordinates).

If $S$ is rational, there exists a smooth weak Del Pezzo surface $\hat{S}$, obtained by the blow-up of six points of $\p^2$ (proper or infinitely near), so that the anti-canonical divisor $-K_{\hat{S}}$ yields a birational morphism $\hat{S}\to S$.
Moreover each curve on $S$ that is not a line corresponds to a divisor equivalent to $kL-\sum_{i=1}^6 a_i E_i$, where $L$ is the pull-back of a line in $\p^2$, the $E_i$ are exceptional divisors and $-K_{\hat{S}}=3L-\sum_{i=1}^6 E_i$ (see \cite[Set-up~4.1]{BL}). We obtain $$6=d=3k-\sum_{i=1}^6 m_i,\ 2=g=\frac{(k-1)(k-2)}{2}-\sum_{i=1}^6 \frac{m_i(m_i-1)}{2}.$$
In particular, $$(3k-6)^2=(\sum m_i)^2\le 6\sum (m_i)^2=6((k-1)(k-2)-4+(3k-6))=6(k^2-8),$$
which implies that $4\le k\le 8$. 
We get finitely possibilities for $(k,m_1,\dots,m_6)$, but each irreducible component corresponds to only one solution. 
For each one, we can order the multiplicities so that $m_1\ge m_2\ge \dots \ge m_6$ and assume that $k\ge m_1+m_2+m_3$ (see \cite[Set-up~4.1]{BL}). 
This gives only one numerical possibility, which is $(k,m_1,\dots,m_6)=(4,2,1,1,1,1,0)$. 
The set of such curves corresponds to quartics with a double point passing through $4$ other given points, and has dimension $7$ (quartics with a double point have dimension $11$, and each simple base point drops the dimension by one).

Now consider the case where $S$ is not rational, that is, $S$ is the cone over a smooth cubic $\Gamma$. 
We now show that there is no smooth curve of type $(2,6)$ on $S$.
By contradiction, assume $C \subset S$ is such a curve. 
By Riemann-Roch's formula (see \cite[Lemma 2.3]{BL}), $C$ is contained in a pencil of cubic surfaces generated by $S$ and another cubic $S'$.
Replacing $S'$ by a general member of the pencil, we can assume that $S'$ is smooth outside of the vertex $p$ of the cone $S$.
Since $S'$ cannot be a cone at this point (because $C \subset S \cap S'$), we obtain that $S'$ is a normal rational cubic. 
Now consider the residual cubic curve $C'$ in $S\cap S'$.
The curve $C'$ is not contained in a plane: otherwise $C \subset S'$ would be linearly equivalent to the complete intersection of a quadric and a cubic, hence equal to such a complete intersection. This is impossible since such curves have genus $4$. 
It follows that the cubic curve $C'$ is a union of rational curves. 
As every morphism from a rational curve to $\Gamma$ is constant, $C'$ is the union of three distinct and  not coplanar lines $l_1,l_2,l_3$ through the vertex $p$. 
Taking the above notation for the desingularisation $\hat{S}'$ of $S'$, the residual cubic is of the form $-3K_{\hat{S}'}-(4L-2E_1-E_2-\dots-E_5)=5L-E_1-2E_2-\dots-2E_5-3E_6$. 
It has arithmetic genus $-1$, hence the configuration of the three lines after blowing-up is as follows: one is disjoint from the two others, which intersect. 
This is impossible: if the vertex is blown-up the three lines become disjoint, and otherwise they all meet in the same point.
\end{proof}

The proof of Theorem~\ref{Thm:WeakFano} is now a matter of putting together what we have done so far:

\begin{proof}[Proof of Theorem~\ref{Thm:WeakFano}]\ \\
\indent~\ref{case:thm1}
 Assume that the blow-up $X$ of~$Y$ along $C$ is weak-Fano. By Proposition~\ref{Prop:NecessaryConditionsPropre}, $(g,d)\in \Lp$ if $C$ is contained in a hyperplane section and $(g,d)\cup \Lq$ otherwise. Moreover, there is no $3$-secant line to $C$ in $Y$ since the strict transform of such a curve would intersect $-K_X$ negatively. By Corollary~\ref{Coro:WeakFano} the curve $C$ is contained in a smooth hyperquadric section and $\lvert -K_X\rvert$ has no base-point.

\ref{case:thm2}\ref{case:thmi} We assume  $(g,d) \in \Lp$. 
By Corollary~\ref{Coro:WeakFano} we get that $X$ is weak-Fano, and in fact Fano in cases $(g,d) = (0,1)$ or $(1,3)$.
Conversely, if $(g,d) \in \Lp \smallsetminus \{(0,1),\, (1,3)\}$, then by Lemma~\ref{Lemm:HRS4.2} we see that $C$ admits a least one $2$-secant line, hence the anticanonical divisor $-K_X$ is not ample.
It remains to study the anticanonical morphism in these cases: This is done by a case by case analysis in \S~\ref{Sec:links} below.

\ref{case:thm2}\ref{case:thmii} We assume $(g,d) \in \Lq$ and that $C$ does not admit any $3$-secant line. 
By Proposition~\ref{Prop:Always3secant}, the curve $C$ is not contained in a hyperplane section.
By Corollary~\ref{Coro:WeakFano} the blow-up $X$ is weak-Fano, and as before Lemma~\ref{Lemm:HRS4.2} ensures that $X$ is not Fano.
Finally, if the anticanonical morphism was divisorial, it would appear as one of the 24 cases of \cite[Theorem~4.9 and Table~A.4]{JPR}, which is not the case.
We conclude that the anticanonical morphism is small in these four cases.

The condition of having no $3$-secant line yields an open subset in the Hilbert scheme parametrising smooth curves of genus~$g$ and degree~$d$ in $Y$. The fact that this set is non-empty is provided by Proposition~\ref{Prop:ExistenceWithoutLines}. 
In cases $(g,d) \in \{(1,4), (0,5), (1,5)\}$, we know from \cite{HRS} that the Hilbert scheme $\mathcal{H}^S_{g,d}(Y)$ parametrising smooth curves of genus~$g$ and degree~$d$ on a smooth cubic threefold $Y\subset \p^4$ is irreducible.
So in these three cases we obtain a dense open subset in $\mathcal{H}^S_{g,d}(Y)$.
The irreducibility of $\mathcal{H}^S_{2,6}(Y)$ seems to be open, but we can still prove that the subset $U\subset \mathcal{H}^S_{2,6}(Y)$ consisting of curves with no $3$-secant is dense. 
By Corollary~\ref{Coro:WeakFano}, every element of $\mathcal{H}^S_{2,6}(Y)$ that corresponds to a curve $C\subset Y$ not contained in a hyperplane is in fact in $U$ (and the converse also holds by Proposition~\ref{Prop:Always3secant}), so it remains to show that the closed subset $V\subset \mathcal{H}^S_{2,6}(Y)$ corresponding to curves in hyperplane sections does not contain any irreducible component of $\mathcal{H}^S_{2,6}(Y)$. 
By a classical result (see e.g. \cite[Proposition 2.1]{HRS}), every irreducible component of $\mathcal{H}^S_{2,6}(Y)$ has dimension at least $12$.
So Lemma \ref{lem:26sing} allows us to conclude.
\end{proof}

\begin{rem}
In the statement of Theorem~\ref{Thm:WeakFano}\ref{case:thm2}\ref{case:thmii}, one could be tempted to replace the condition ``without a 3-secant line in $Y$'' by the condition ``not contained in a hyperplane section'', but there is a subtlety here in the case $(g,d) = (0,5)$, which comes from Remark~\ref{Rem:cas05}.
\end{rem}

\subsection{Sarkisov links}\label{Sec:links}
In this section, we describe the Sarkisov links provided by Theorem~\ref{Thm:WeakFano}. 
The summary of what we obtain is given in Table~\ref{tab:bigtable}.

{\footnotesize
\begin{table}[h]
\begin{center}
\begin{tabular}{ccClCl}
\toprule
List & Properties & (g,d) & Sarkisov &  \text{\# 2-secant} & Reference(s) \\
&of $X$&  & link &  \text{lines} &  \\
\midrule
\multirow{6}{*}{$\Lp$}& \multirow{2}{*}{Fano} & (0,1) & conic bundle & 1 \text{ (itself)} & \cite[\S 12.3, No 11]{IP}\\
&      & (1,3) & DP3 fibr. & 0 & \cite[\S 12.3, No 5]{IP} \\
\cmidrule{2-6}
& \multirow{2}{*}{\begin{tabular}{c}weak-Fano \\ divisorial \end{tabular}} & (1,4) & - & 10 & \cite[Tab.A.4, n$^\circ$6]{JPR}\\
& & (4,6) & - & 27 & \cite[Tab.A.4, n$^\circ$25]{JPR} \\
\cmidrule{2-6}
& \multirow{2}{*}{\begin{tabular}{c}weak-Fano \\ small \end{tabular}} & (0,2) & DP4 fibr. & 1 & \cite[\S 7.4]{JPR2}\\
& & (0,3) & terminal Fano & 6 & \cite[Tab.~6(4)]{CM} \\ 
\midrule
\multirow{6}{*}{$\Lq$}& \multirow{6}{*}{\begin{tabular}{c}weak-Fano \\ small \end{tabular}} & \multirow{2}{*}{(1,4)} & \multirow{2}{*}{point in $V_{14}$} & \multirow{2}{*}{16} & \cite[Tab.~4(2)]{CM}  \\
& & & & & \cite[\S(2.8)]{T}\\
& & \multirow{2}{*}{(1,5)} & \multirow{2}{*}{curve in $V_{14}$} & \multirow{2}{*}{25} & \cite[Tab.~1(63)]{CM}  \\
& & & & & \cite[\S III.1 p.~858]{Isk80}\\
& & (0,5) & back to $Y$ & 31 & \cite[Tab.~1(29)]{CM}\\
& & (2,6) & back to $Y$ & 39 & \cite[Tab.~1(33)]{CM} \\
\bottomrule\\
\end{tabular}
\end{center}
\caption{The Sarkisov links for the types of $\Lp\cup\Lq$. 
\small{({\it Small and divisorial correspond to the anticanonical morphism, and DP$n$ fibr.~is a fibration whose general fibre is a Del Pezzo of degree $n$.})}}
\label{tab:bigtable}
\end{table}
}

Recall the following result; we refer to our previous paper \cite[\S 2.1]{BL} for details.

\begin{prop}\label{Prop:SarkisovLink}
Assume that $X$ is a smooth threefold with Picard number~$2$,  big and nef anticanonical divisor, and small anticanonical morphism. 
Then the two contractions on $X$ yield a Sarkisov link:
\begin{equation*}\label{eq:link}
\xymatrix{
& X \ar@{-->}[r] \ar[dl] & X' \ar[dr] \\
Y &&& Y'.
}
\end{equation*}
\end{prop}
In the previous diagram $X \dashrightarrow X'$ is an isomorphism or a flop, depending if $-K_X$ is ample or not.  
In our situation, we know one of the contraction, namely $X \to Y$ which is the blow-up of a smooth curve $C$. 
On the other hand there are several possibilities for the contraction $X' \to Y'$. 
It can be divisorial, and in this case $Y'$ is again a Fano threefold with Picard number~$1$, and possibly with a terminal singularity.
The contraction can also be a fibration, either a conic bundle, or a fibration in Del Pezzo surfaces.
Finally observe that if the anticanonical morphism on $X$ is divisorial, then there is no such Sarkisov link.

We now describe the Sarkisov links associated with the curves listed in Theorem~\ref{Thm:WeakFano} (see Table \ref{tab:bigtable}). 
We give elementary arguments whenever possible, but for the most delicate cases we have to rely on previous classification results.

We saw in Proposition~\ref{Prop:Fano} that a line (case $(0,1)$) gives a Fano threefold~$X$ with a conic bundle structure, and a plane cubic (case $(1,3)$) gives a Fano threefold with a Del Pezzo fibration of degree 3.
These correspond respectively to No 11 and 5 in the Table \cite[\S 12.3]{IP}.

Now consider the case $(0,2)$, that is, $C$ is a smooth conic.
Any curve $\Gamma$ of degree $n$ that is $n$-secant to $C$ must be contained in the base locus of the pencil of hyperplane sections containing $C$, hence $\Gamma$ must be the unique 2-secant line to $C$.
Hence the anticanonical morphism is small, contracting only the transform of this 2-secant line.
After flopping this curve we obtain a weak-Fano threefold $X'$ that admits a Del Pezzo fibration of degree 4, in accordance with \cite[\S 7.4]{JPR2}.

In the case $(0,3)$, the curve $C$ is contained in a unique hyperplane $H \subset \p^4$.
Let $\Gamma$ be a curve of degree $n$ that is $n$-secant to $C$, then on $X$ its strict transform satisfies $-K_X \cdot \tilde \Gamma = 0$. 
We can find a pencil of members of $\lvert -K_X \rvert$ containing $\tilde \Gamma$, which correspond on $H$ to a pencil of quadric surfaces containing $C$: This shows that $\Gamma$ must be the residual line of the pencil.
In particular the only curves contracted by the anticanonical morphism are the transforms of the six 2-secant lines given by Lemma~\ref{Lemm:HRS4.2}.  
After flopping these curves, we obtain a weak-Fano threefold $X'$, which according to \cite[Table 6 (4)]{CM} admits a divisorial extremal contraction to a terminal Fano threefold.

In the case of a curve $C$ of type $(g,d) = (1,4)$, contained in a hyperplane $H$, there exists a pencil of smooth quadric surfaces in $H$ that contain $C$. 
Intersecting with the cubic threefold~$Y$ we obtain a pencil of residual conics, which are 4-secant to $C$.
The transforms of these conics on $X$ are trivial against the canonical divisor.
In conclusion the anticanonical map after blow-up is divisorial, contracting the cubic surface $H \cap Y$ on a curve: see \cite[Table A.4, no 6]{JPR}.


The case $(4,6)$ always corresponds to the complete intersection of a hyperplane and a quadric, by Lemma~\ref{Lemm:hypersurfaces}. 
As a consequence any curve of degree $n$ in the cubic surface containing $C$ is $2n$-secant to $C$, and the anticanonical map after blow-up is divisorial \cite[Table A.4, no 25]{JPR}. 

The case $(g,d) = (1,5)$ is given as an open case in \cite[Proposition 6.5, no 8]{JPR2}, with a hypothetical link to a Del Pezzo fibration of degree 5.
However it was proved in \cite[\S III.1 p.~858]{Isk80} that the blow-up of a smooth normal elliptic quintic always yields a link to a Fano 3-fold $V_{14}$ of genus 8.

Finally, cases $(1,4), (0,5)$ and $(2,6)$ do not appear in \cite{JPR2}, so we conclude that after blowing-up $C$ and performing a flop, we obtain a weak-Fano threefold $X'$ with a divisorial extremal contraction.
Consulting the tables in \cite{CM} we see that this contraction must be a contraction to a smooth point in $V_{14}$ in the case $(1,4)$ (this construction was already noticed in \cite[\S(2.8)]{T}), and a contraction to a curve of the same type in a cubic threefold in case $(0,5)$ and $(2,6)$.
It turns out that the cubic threefold we obtain is isomorphic to the one we started with. 
In fact we shall see in Proposition~\ref{Prop:TauOnBlowUp} that in these two cases the Sarkisov link can be seen as an involution of~$Y$.
\section{Examples of birational selfmaps}\label{Sec:Examples}

In this section we produce special examples of birational selfmaps, either on a smooth cubic threefold, or on a blow-up of such a threefold. 
We are in particular interested in producing maps contracting arbitrary ruled surfaces, or pseudo-automorphisms with dynamical degree greater than 1 on a threefold with low Picard number.
We start by recalling the basic notions we shall need.

\subsection{Basics}\label{Sec:BasicOfLastSec}
We recall some basic notions on the genus of a birational map, dynamical degrees, involutions and the non-existence of a $\mathbb{G}_a$-action on a non-rational unirational threefold.

\subsubsection{Genus of a birational map on a threefold}
We recall the notion of genus of a birational map between threefold, as introduced by Frumkin \cite{Fru};  
see also \cite{LamyGenus}.
If $X$ is a smooth threefold, every surface contracted by an element $\varphi\in \Bir(X)$ is birational to $C\times \p^1$, for some smooth projective curve $C$. The \emph{genus} of $\varphi$ is by definition the highest possible genus of $C$ obtained, when considering all surfaces contracted by $\varphi$.
By convention we declare the genus to be equal to $-\infty$ when no surface is contracted, that is, in the case of a pseudo-automorphism.

For each $g$, the set $\Bir(X)_{\le g}$ of all birational maps of genus $\le g$ is a subgroup of $\Bir(X)$.

It is easy to obtain elements of $\Bir(\p^3)$ of any genus: in fact it is sufficient to consider Jonqui\`eres elements.
This shows in particular that $\Bir(\p^3)$ is not generated by automorphisms and finitely many other elements (see \cite{Pan}). 
We will see that such a phenomenon also holds for a smooth cubic threefold.

\subsubsection{Dynamical degree}

Let $f\colon X \dashrightarrow X$ be a dominant rational map on a projective variety of dimension $n$, and let $L$ be an ample divisor on $X$.
We define the degree of $f$ by $\deg f = (f^*L \cdot L^{n-1})/(L^n)$, and the (first) dynamical degree of $f$ by
$$\lambda_1(f) = \lim_{m \to \infty} \deg(f^m)^{1/m}.$$ 
The dynamical degree does not depend on the choice of $L$ and is invariant by conjugation (see \cite[Corollary 7]{DS}): if $\phi\colon Y \dashrightarrow X$ is a birational map, then 
$$\lambda_1(f) = \lambda_1(\phi^{-1} \circ f \circ \phi).$$
If the action of $f$ on $\Pic(X)$ satisfies $(f^*)^n = (f^n)^*$ one says that $f$ is algebraically stable, and in this case the dynamical degree $\lambda_1(f)$ is equal to spectral radius of $f^*$.
This is the case in particular if $f$ is an automorphism, or a pseudo-automorphism.
When $f$ is birational we often prefer to use the action by pushforward $f_* = (f^{-1})^*$.
In general $\lambda_1(f) \neq \lambda_1(f^{-1})$, but this is the case if $f^{-1}$ is conjugate to $f$, which is true in particular when $f$ is the composition of two involutions. 

\subsubsection{Involutions}

Our examples will be obtained as family of involutions, or as composition of involutions.
We recall here a few well-known constructions for future reference.

\begin{enumerate}[wide]
\item \label{inv P1 exchanges} If $a,b,c \in \p^1$ are three distinct points, there exists a unique involution of $\p^1$ that exchanges $a$ and $b$ and fixes $c$.
Indeed up to the action of $\PGL_2$ we can assume $a = 0$, $b = \infty$ and $c = 1$, and the involution is then $z \mapsto \frac1z$.
\item \label{inv P1 fixes} Similarly given $a,b \in \p^1$ two distinct points, there is a unique involution that fixes $a$ and $b$.
If we assume $a = 0$ and $b = \infty$, the involution is $z\mapsto -z$.
\item \label{inv S exchanges} Now if $S \subset \p^3$ is a smooth cubic surface, and $c \in S$ is a point, we define the Geiser involution centred in $c$ as follows.
Given a general line $L$ through $c$, we have $L \cap S = \{a,b,c\}$ with $a,b,c$ distinct, and the Geiser involution restricted to $L$ is by definition the unique involution of $L$ fixing $c$ and exchanging $a$ and $b$.  
This gives a birational involution of $\p^3$, which restricts to the classical Geiser involution on~$S$: the blow-up $X\to S$ of $c$ is a del Pezzo surface of degree $2$, and the lift of the involution is the involution associated to the double covering $X\to \p^2$ given by $\lvert -K_X\rvert$. In particular, the exceptional divisor is exchanged with the strict transform of the hyperplane section tangent at $c$.
\item \label{inv S fixes} In the previous setting, one can also define an involution of $\p^3$ that fixes pointwise the surface $S$, by defining the restriction on $L$ to be the involution fixing $a$ and $b$.
\end{enumerate}

The constructions in~\ref{inv S exchanges} and~\ref{inv S fixes} can be generalised for a cubic hypersurface in $\p^n$ for an arbitrary $n \ge 2$.
In particular the construction~\ref{inv S fixes} was the building block for the examples obtained in \cite{blanc:2013}.

\subsubsection{$\mathbb{G}_a$-action}

Here, as a side remark, we give a proof of the following folklore result, which we mentioned in the introduction, and which does not seem to be available in the literature. 
We learnt the argument from D.~Daigle.

\begin{prop}\label{Prop:Daigle}
Let $Y$ be a threefold, which is unirational but not rational.
Then any rational $\mathbb{G}_a$-action on $Y$ is trivial. 
\end{prop}

\begin{proof} 
Assume the existence of a non-trivial rational $\mathbb{G}_a$ action on $Y$. Replacing $Y$ by another projective smooth model, we can assume that the action is regular. 
Then there exists an open set $U \subset Y$ where the action is a translation, which corresponds to the existence of an equivariant isomorphism $U\to \mathbb{G}_a\times V$, where the action on $V$ is trivial and the action on $\mathbb{G}_a$ is the translation (follows from the work of Rosenlicht \cite{Ros}).
Since $Y$ is unirational, the variety $V$ is a unirational surface $S$, which is thus rational. This implies that $U$ is rational: contradiction.
\end{proof}
The above result shows that there is no rational $\mathbb{G}_a$-action on a smooth cubic threefold, using the non-rationality result of Clemens and Griffiths \cite{CG}.

\subsection{Birational selfmaps of a cubic threefold with arbitrary genus}
\label{Sec:Genus}
In \cite[Question 11]{LamyGenus} the question is asked whether there exists a birational map on a smooth cubic threefold with genus $\ge 1$.
The existence of Sarkisov links blowing-up a curve of type $(2,6)$ (see \S\ref{Sec:links}) already shows that the answer is affirmative. 
In this section we give two other constructions: a very simple one that produces examples of genus~$1$, and then a more elaborate one that yields maps of arbitrary genus.

Our first construction is based on the fibration associated with an elliptic curve on a smooth cubic threefold.

\begin{prop}
Let $Y\subset \p^4$ be a smooth cubic threefold, let $C\subset Y$ be a smooth plane cubic curve, and let $\pi\colon Y\dasharrow \p^1$ the projection from $C$ $($or more precisely from the plane containing it$)$.

We denote by $\Bir(Y/\pi)$ the subgroup of $\Bir(Y)$ of elements $\varphi$ such that $\pi\varphi=\pi$. Then, there exist elements of $\Bir(Y/\pi)$ having genus~$1$ and dynamical degree~$>1$.
\end{prop}

\begin{proof}
We associate an element $\varphi_L\in \Bir(Y/\pi)$ to any line $L\subset Y$ disjoint from $C$, by considering a one-parameter family of Geiser involutions. 
Let $t\in \p^1$ be a general point. 
The corresponding fibre $X_t=\pi^{-1}(t)$ is a smooth cubic surface with a marked point $p_t=L\cap X_t$.
We define  $\varphi_L$ as the birational map whose restriction to $X_t$ is the Geiser involution associated with $p_t$. Note that the Geiser involution on $X_t$ exchanges the curve $C$ with another curve $C_t\subset X_t$, which is birational to $C$. The union of all curves $C_t$ covers a surface $V\subset Y$ that is birational to $C\times \p^1$, and which is contracted by $\varphi_L$ onto $C$. The Geiser involution of $X_t$ contracts the curve $\Gamma_t$ which is the hyperplane section tangent at $c$. The union of these curves covers a surface that is rational.
Since all other surfaces contracted by $\varphi_L$ are contained in special fibres, these are rational or equal to a cone over an elliptic curve, so we obtain that the genus of $\varphi_L$ is $1$.

Choosing general distinct lines $L_1,L_2,L_3$ on $Y$, we obtain involutions $\sigma_1,\sigma_2,\sigma_3\in \Bir(X/\pi)$. We claim that $\sigma_3\sigma_2\sigma_1$ has dynamical degree $>1$ and genus $1$. To show this, we take a general fibre $X$ and consider the restrictions $\hat{\sigma}_1,\hat{\sigma}_2,\hat{\sigma}_3\in \Bir(X)$ of $\sigma_1,\sigma_2,\sigma_3$. 
The lift of $\hat{\sigma}_i$ to the blow-up $Z_i\to X$ of the point $p_i=L_i\cap X$ is an automorphism, which sends the exceptional divisor $E_i$ onto $-K_X-2E_i$, where $K_X$ denotes the pull-back of the anti-canonical divisor. Since the map is of order $2$ and preserves the anti-canonical divisor $-K_X-2E_i$, the action relative to the basis $(-K_X,E_i)$ is then
$$\left(\begin{array}{rr} 2 & 1\\ -3 & -2\end{array}\right).$$
Hence, every element of $\langle\hat{\sigma}_1,\hat{\sigma}_2,\hat{\sigma}_3\rangle$ sends $-K_X$ onto a linear system equivalent to $-dK_X$ for some integer $d$, which corresponds to the degree of the map, according to the ample divisor $-K_X$.

To simplify the notation, we define $p_i=p_{i-3}$, $\hat{\sigma}_i=\hat{\sigma}_{i-3}$ for $i\ge 4$. The lines $L_1,L_2,L_3$ being general, we can assume that $p_{i+3}=p_{i}$ is not equal to $p_{i+2}$, $\hat{\sigma}_{i+2}(p_{i+1}),\hat{\sigma}_{i+2}\hat{\sigma}_{i+1}(p_{i})$ for $i=1,2,3$, and then for each $i\ge 1$.

Writing $\rho=\frac{3}{2}$, we prove then, by induction on $k$, that 
$$R_k=\hat{\sigma}_{k}\hat{\sigma}_{k-1}\dots \hat{\sigma}_{1}(-K_X)$$
has degree $d_k\ge \rho^k$ and multiplicity at most $\rho d_k$ at $p_k$. Moreover, we also show that $R_k$ has multiplicity at most $d_k,\frac{1}{\rho} d_k$ at $\hat{\sigma}_k(p_{k-1}),\hat{\sigma}_k\hat{\sigma}_{k-1}(p_{k-2})$ (if $k\ge 2$, respectively $k\ge 3$) and at most multiplicity $\frac{1}{\rho^2}d_k$ at all other points.

This result is true for $k=1$, since $R_1$ has degree $2$, multiplicity $3$ at $p_1$ and no other base-point. We then use the matrix above to compute $R_{k+1}=\hat{\sigma}_{a_{k+1}}(R_k)$, and obtain that $R_{k+1}$ has degree $d_{k+1}=2d_k-m$ and multiplicity $3d_k-2m\le3d_k-\frac{3}{2} m= \rho d_{k+1}$ at $p_{k+1}$, where $m$ is the multiplicity of $R_k$ at $p_{k+1}$. 

By hypothesis, $p_{k+1}$ is not equal to $p_k$, $\hat{\sigma}_k(p_{k-1}),\hat{\sigma}_k\hat{\sigma}_{k-1}(p_{k-2})$, which implies that $m\le \frac{1}{\rho^2}d_k$ and thus that
$d_{k+1}=2d_k-m\ge (2-\frac{1}{\rho^2})d_k\ge \rho d_k$.

Moreover, the multiplicity of $R_{k+1}=\hat{\sigma}_{k+1}(R_k)$ at $\hat{\sigma}_{k+1}(p_k)$ is the multiplicity of $R_k$ at $p_k$, which is at most $\rho d_k\le \frac{1}{\rho}d_{k+1}$. Similarly, the multiplicity of $R_{k+1}$ at $\hat{\sigma}_{k+1}\hat{\sigma}_k(p_{k-1})$ is at most $\frac{1}{\rho^2} d_{k_1}$ and all other base-points have multiplicity at most $\frac{1}{\rho^3} d_{k_1}$.

This gives the claim, which implies that $(\hat{\sigma}_3\hat{\sigma}_2\hat{\sigma}_1)^i$ has degree at least $(\frac{3}{2})^{3i}$ and implies that $\hat{\sigma}_3\hat{\sigma}_2\hat{\sigma}_1$, and thus ${\sigma}_3{\sigma}_2{\sigma}_1$, has dynamical degree $>1$.
\end{proof}

\begin{rem}
\begin{enumerate}[wide]
\item Taking other smooth cubic curves in $Y$, we can obtain all types of elliptic curves, and compose the maps obtained to obtain birational maps that contract arbitrary many surfaces that are not pairwise birational.

\item It seems plausible that we could obtain maps of higher genus by generalising the above construction, using a family of Bertini involutions associated with an hyperelliptic curve. 
However, finding the hyperelliptic curves does not seem to be easy.
Besides, we have another construction which provides all possible curves instead of only hyperelliptic ones (see Proposition~\ref{Prop:Eachsurface} below).
\end{enumerate}
\end{rem}

Now we construct a class of birational involutions on $Y$ with arbitrary genus, and in fact contracting any given class of  ruled surface. This will be done in the group preserving the fibration associated to the projection from a line of~$Y$, which yields a conic bundle structure on the blow-up $\hat{Y}\to Y$ of the line. 

The following two results on conic bundles will provide the class of involutions.

\begin{lemm}\label{Lemm:ConicBundleInvolution}
Let $\pi\colon Q\to B$ be a conic bundle over an irreducible algebraic variety $B$, given by the restriction of a $\p^2$-bundle $\hat{\pi}\colon P\to B$.

Let $s\colon B\to P$ be a section, whose image is not contained in $Q$. We define $\iota\in \Bir(Q/\pi)$ to be the birational involution whose restriction to a general fibre $\pi^{-1}(p)$ is the involution induced by the projection from the point $s(p)$: it is the involution of the plane $\hat{\pi}^{-1}(p)$ which preserves the conic $\pi^{-1}(p)$ and fixes $s(p)$.

If $\Gamma\subset B$ in an irreducible hypersurface which is not contained in the discriminant locus of $\pi$ and such that $s(\Gamma)\subset Q$, the hypersurface $V=\pi^{-1}(\Gamma)$ of $Q$ is contracted onto the codimension $2$ subset $s(\Gamma)$.
\end{lemm}

\begin{proof}
We choose a dense open subset of $B$ which intersects $\Gamma$ and trivialises the  $\p^2$-bundle, and apply a birational map to view $Q$ inside of $\p^2 \times B$, given by $F\in \C(B)[x,y,z]$, homogeneous of degree $2$ in $x,y,z$. The  fibre of $\pi \colon Q\to B$ over a general point of $\Gamma$ (respectively of $B$) is a smooth conic. The section $s$ corresponds to $[\alpha:\beta:\gamma]$, where $\alpha,\beta,\gamma\in \C(B)$.

We denote by $f=F(s)\in \C(B)$ the evaluation of $F$ at $x=\alpha,y=\beta,z=\gamma$ (which looks like an evaluation at $s$ but depends in fact of $(\alpha,\beta,\gamma)\in \C(B)^3$), and by $f_x,f_y,f_z$ the evaluation of the partial derivatives of $F$ at $x=\alpha,y=\beta,z=\gamma$ (same remark). Writing then $R=x f_x+yf_y+zf_z$, we claim that $\iota$ is given by 
$$\iota\colon [x:y:z]\mapsto [\alpha R-x f:\beta R-y f:\gamma R-zf].$$
(Now the map only depends on $s$ and not on the choice of $\alpha,\beta,\gamma$). This claim will imply the result: for a general point of $p\in \Gamma$ the value of $f$ is zero, and the corresponding rational transformation of $\p^2\sim \hat{\pi}^{-1}(p)$ contracts the whole plane onto $s(p)$.

It remains to see that $\iota$ has the desired form. Using the Euler formula $f=\frac{1}{2}(\alpha f_x+\beta f_y+\gamma f_z)$, the above transformation corresponds to the element $$M=\left(\begin{array}{lll}
f_x \alpha-f_y \beta-f_z \gamma&2 \alpha f_y&2\alpha f_z\\
2\beta f_x&-f_x\alpha+f_y\beta-f_z\gamma&2\beta f_z\\
2\gamma f_x&2\gamma f_y&-f_x\alpha-f_y\beta+f_z\gamma\end{array}
\right)$$
of $\PGL_3(\C(B))$, whose square is the identity (or more precisely $(f)^2$ times the identity). It is then a birational involution of $P$. Moreover, $M$ fixes $[\alpha:\beta:\gamma]$ (multiplying the matrices we get $[\alpha f:\beta f:\gamma f]$).  It remains then to see that the above map preserves the equation of $F$. This can be done explicitly by writing $F=ax^2+by^2+cz^2+dxy+exz+fyz$, where $a,b,c,d,e,f\in \C(\Gamma)$.
\end{proof}
\begin{lemm}\label{Lemm:ExtensionSection}
Let $\pi\colon Q\to \p^2$ be a conic bundle, given by the restriction of a $\p^2$-bundle $\hat{\pi}\colon P\to \p^2$. Let $\Gamma\subset \p^2$ be an irreducible curve, not contained in the discriminant of $P$. Then, there is a section $s\colon \p^2\to P$ of $\hat{\pi}$, whose image is not contained in $Q$ but such that $s(\Gamma)\subset Q$.
\end{lemm}
\begin{proof}
The preimage by $\pi$ of $\Gamma$ is a surface $\hat S$, and there exists a section $s_0\colon \Gamma \to \hat S$ by Tsen's theorem (see \cite[Corollary 6.6.2 p.~232]{Ko}). It remains to see that we can extend $s_0$ to a section $s\colon \p^2\to P$ of $\hat{\pi}$, whose image is not contained in $Q$. 

Taking local coordinates, we can view $\hat{\pi}$ as the projection $\A^4\to \A^2$ on the first two coordinates and $\Gamma$ as curve in $\A^2$ given by some irreducible polynomial $P\in \C[x,y]$. Then,  $s_0$ is given by two rational functions $\frac{f_1}{g_1},\frac{f_2}{g_2}$, where $f_1,f_2,g_1,g_2\in \C[\Gamma]=\C[x,y]/(P)$ and $g_1,g_2\not=0$. There are then plenty of ways to extend $s_0$, by choosing representatives of the $f_i$ and $g_i$ in $\C[x,y]=\C[\A^2]$.
\end{proof}

\begin{prop}\label{Prop:Eachsurface}
Let $Y\subset \p^3$ be a smooth cubic hypersurface, let $\ell\subset Y$ be a line and let $\Gamma$ be an abstract irreducible curve.

Then, there exists a birational involution $\iota\in \Bir(Y)$ that preserves a general fibre of the projection $Y\dasharrow \p^2$ away from $\ell$, and which contracts a surface birational to $\Gamma\times \p^1$.
\end{prop}

\begin{proof}
Blowing-up $\ell$, we obtain a conic bundle $\hat{Y}\to \p^2$. 
We then pick an irreducible curve in $\p^2$, which is birational to $\Gamma$ and not contained in the discriminant locus and apply Lemmas~\ref{Lemm:ConicBundleInvolution} and~\ref{Lemm:ExtensionSection}.
\end{proof}

\begin{rem}
As before, taking different lines and composing maps, one obtain maps of dynamical degree $>1$ contracting any possible birational type of ruled surfaces.
\end{rem}

\subsection{Pseudo-automorphisms}\label{Sec:Pseudoauto}

The aim of this section is to prove Proposition~\ref{Prop:mainC}, i.e.~to produce a pseudo-automorphism with dynamical degree greater than $1$ on a smooth threefold $Z$ with Picard number~$3$.
Observe that this is the minimal value of the Picard number: such a pseudo-automorphism induces a linear map on $\Pic(Z)$ with determinant $\pm 1$, one eigenvalue equal to~$1$ (the canonical divisor is preserved) and one eigenvalue of modulus bigger than~$1$ (equal to the dynamical degree).
This answers, at least in dimension~$3$, a question asked in \cite[Question 1.5]{blanc:2013}.\\

The proof follows directly from the three propositions that occupy the rest of this section.
The first step relies on the fact that $(-K_X)^3 = 2$, which is true for a curve of type $(2,6)$ but also for a curve of type $(0,5)$.

\begin{prop}\label{Prop:TauOnBlowUp}
Let $Y\subset \p^4$ be a smooth cubic threefold.
Let $C\subset Y$ be a smooth curve of genus~$g$ and degree~$d$, with $(g,d)\in \{(0,5),(2,6)\}$, which does not admit any $3$-secant line in $Y$. 
Let $\pi\colon X\to Y$ be the blow-up of~$Y$ along $C$, $E=\pi^{-1}(C)$ be the exceptional divisor and $H_X\in \Pic(X)$ be the pull-back of a hyperplane section of~$Y$. 

Then, the linear system $\lvert -K_X \rvert$ yields a surjective morphism $\sigma\colon X\to \p^3$ with general fiber equal to two points. The involution corresponding to the exchange of the two points is a pseudo-automorphism $\tau\colon X\dasharrow X$ whose action on $\Pic(X)$, with respect to the basis $(H_X,E)$, is respectively
$$\begin{pmatrix} 
  13 & 24 \\ -7 & -13
  \end{pmatrix},\begin{pmatrix} 
  11 & 20 \\ -6 & -11
  \end{pmatrix}$$
  for $(g,d)=(0,5)$ and $(g,d)=(2,6)$.
\end{prop}
\begin{proof}
By Theorem~\ref{Thm:WeakFano}\ref{case:thm2}\ref{case:thmii}, the threefold $X$ is weak-Fano.
By Lemma~\ref{Lemm:K^3} and Proposition~\ref{Prop:-K irreducible}, we have 
$$(-K_X)^3 = 2 \text{ and } \dim \lvert -K_X \rvert = 3.$$

By Proposition~\ref{Prop:smoothgeneral}, the linear system $\Lambda$ of quadric hypersections of~$Y$ through $C$ has no base-point outside $C$ and has a general member which is smooth. Moreover, the linear system $\lvert -K_X \rvert$ is base-point-free.

So the rational map induced by $\lvert -K_X \rvert$ is a morphism $\sigma\colon X\to \p^3$. 
It is surjective, otherwise the image would have at most dimension 2 and we would have $(-K_X)^3 = 0$, contradiction. 
Moreover, since by Theorem~\ref{Thm:WeakFano} the anticanonical morphism is small, there are finitely many curves that intersect $K_X$ trivially, hence finitely many fibres that have positive dimension. 
The number of points in a finite fibre can be computed as the intersection of three elements of $\lvert -K_X \rvert$. 
Since $(-K_X)^3 = 2$, a general fibre consists of $2$ points, and some codimension $1$ subset of $\p^3$ yields fibres with one point (ramification divisor).

The birational involution $\tau\in \Bir(X)$ that exchanges the two points in a general fibre of $\sigma$ is thus an automorphism outside of the finite set of curves having zero intersection with $K_X$, and is then a pseudo-automorphism. 
Let $f\in X$ be a general fibre of the $\p^1$-bundle $\pi|_E\colon E\to C$, and $\ell\in X$ be the pull-back by $\pi$ of a line in $Y$. 
The dual of $\Pic(X)$ is then generated by $\ell,f$, which have intersection $(1,0)$ and $(0,-1)$ with $(H_X,E)$ respectively. 
The pull-back by $\sigma$ of a general line of $\p^3$ is equivalent to $(-K_X)^2=(2H_X-E)^2$.

By Lemma~\ref{Lemm:K^3} we have $(-K_X)^3 =22-4d+2g=2$ and
$K_X^2\cdot E = 2+2d-2g=22-2d$, which yields $$(K_X)^2\cdot H_X=\frac{1}{2}(K_X)^2\cdot (-K_X+E)=12-d,$$
and thus $(-K_X)^2=(12-d)\ell-(22-2d)f$.

Observe that $f\cdot (-K_X)=1$, which implies that $f$ intersects a general fibre of hyperplane into one point, transversally. In particular, $\sigma(f)$ is a line of $\p^3$ and the restriction of $\sigma$ yields an isomorphism $f\to \sigma(f)$. Since $\sigma^{-1}(\sigma(f))$ is equivalent to $(-K_X)^2=(12-d)\ell-(22-2d)f$, we get that $\sigma(f)$ is equivalent to $(12-d)\ell-(23-2d)f$. 

In particular, the action of $\sigma$ on $\Pic(X)$ is then of order $2$, fixing $-K_X$. With respect to the basis $(H_X,-K_X)$, the map is thus of the form
$$\begin{pmatrix} 
  -1 & 0 \\ a & 1
  \end{pmatrix},$$
  for some $a\in \mathbb{Z}$. It remains to compute
  $$a=f\cdot (-H_X+a(-K_X))=f\cdot \sigma(H_X)=\sigma(f)\cdot H_X=12-d$$
  to get the action.
\end{proof}

\begin{prop}\label{Prop:Pairof26}
Let $C_1$ be a curve of genus $2$ and degree $6$ on a smooth cubic threefold~$Y$, which is general in the sense of 
Theorem~\ref{Thm:WeakFano}\ref{case:thm2}\ref{case:thmii}.
Let $\Lambda$ be a general pencil of hyperquadric sections containing~$C_1$. 
Then the base locus of $\Lambda$ is equal to $C_1 \cup C_2$ where $C_2$ is another general smooth curve of type $(2,6)$.
\end{prop}

\begin{proof}
By Proposition~\ref{Prop:smoothgeneral}, a general member $S$ of $\Lambda$ is smooth. 
After blow-up of $C_1$, the pencil of residual curve $C_2$ corresponds to the restriction to $\tilde S$ of a subpencil of the linear system $\lvert-K_X\rvert$.
We know from Proposition~\ref{Prop:TauOnBlowUp} that $\lvert-K_X\rvert$ induces a surjective morphism $X\to\p^3$, hence $\lvert-K_X\rvert_S$ has no base-point. The fact that the image of $\lvert-K_X\rvert$ is not a curve implies, by Bertini's Theorem (in its strong form, see e.g.~\cite[Ex.~III.11.3]{Har} or \cite[Theorem~5.3]{Kleiman}) that a general element of $\lvert-K_X\rvert_S$ is an irreducible smooth curve.

It remains to see that $C_2$ has type $(2,6)$ and is general. 
One can check this using the fact that $C_1$ and $C_2$ are embedded in a smooth Del Pezzo surfaces of degree 4. 
In the notation of Proposition~\ref{Prop:DelPezzo4}, $C_1$ comes from a curve on $\p^2$ of degree $4$ and multiplicities $(2,1,1,1,1)$, and $C_1 \cup C_2$  from a curve of degree $9$ and multiplicities $(3,3,3,3,3)$. 
Thus $C_2$ comes from a curve of degree $5$ and multiplicities $(1,2,2,2,2)$, which yields that $C_2$ has genus 2 and degree 6.
Finally $C_2$ does not admit any 3-secant line, because such a line should be in the base locus of $\Lambda$; so $C_2$ is general in the sense of Theorem~\ref{Thm:WeakFano}\ref{case:thm2}\ref{case:thmii}.
\end{proof}

We can now obtain the proof of Proposition~\ref{Prop:mainC}:

\begin{prop}
Let $Y\subset \p^4$ be a smooth cubic threefold. Let $Q_1,Q_2\subset Y$ be two hyperquadric sections such that $Q_1\cap Q_2=C_1\cup C_2$, where $C_1,C_2$ are two smooth curves of genus $2$ and degree $6$ $($see Proposition~$\ref{Prop:Pairof26})$.

Let $\pi\colon Z\to Y$ be the blow-up of $C_1\subset Y$, followed by the blow-up of the strict transform of $C_2$. Then, $Z$ is a smooth threefold of Picard rank $3$ that admits a pseudo-automorphism of dynamical degree $49+20\sqrt{6}$.
\end{prop}
\begin{proof}
For $i=1,2$, we denote by $\pi_i\colon X_i\to Y$ the blow-up of $C_i$, and by $\tau_i\in \Bir(X_i)$ the birational involution given by Proposition~\ref{Prop:TauOnBlowUp}, which is a pseudo-automorphism. Let us observe that the strict transform $\tilde{C}_{3-i}\subset X_i$ of $C_{3-i}$ on $X_i$ is equal to the intersection of the strict transforms of $H_1$ and $H_2$. Hence, $\tilde{C}_{3-i}$ is the fibre of a line by the birational morphism $\sigma_i\colon X\to \p^3$ given by $|-K_{X_i}|$. In particular, it is invariant by $\tau_i$. This implies that $\tau_i$ lifts to a pseudo-automorphism on the blow-up $\pi_i'\colon Z_i\to X_i$ of $\tilde{C}_{3-i}$.

Note that $\pi_1\circ \pi_i'$ is equal to $\pi\colon Z\to \p^3$, up to an isomorphism $Z\to Z_1$, and that $\pi_2\circ \pi_2'$ is equal to $\pi\colon Z\to \p^3$, up to pseudo-isomorphism $Z\dasharrow Z_2$ (which is an isomorphism outside of the pull-back of $C_1\cap C_2\subset Y$). In particular, the maps $\tau_1,\tau_2$ yield two pseudo-automorphisms $\tau_1',\tau_2'$ of $Z$.

Finally we compute the dynamical degree of the pseudo-automorphism $\tau_1' \circ \tau_2'$. 
We express the action of $\tau_1'$ and $\tau_2'$ on $\Pic(Z)$ by choosing the natural basis $H, E_1, E_2$ (pull-back of hyperplane and exceptional divisors). 
We compute:
\begin{equation*}
\tau_{1*}' \circ \tau_{2*}' = \begin{pmatrix} 11 & 20 & 0 \\ -6 & -11 & 0 \\ 0 & 0 & 1 \end{pmatrix}
\begin{pmatrix} 11 & 0 & 20 \\ 0 & 1 & 0 \\ -6 & 0 & -11 \end{pmatrix}
= \begin{pmatrix} 121 & 20 & 220 \\ -66 & -11 & -120 \\ -6 & 0 & -11 \end{pmatrix}.
\end{equation*}
On checks that $\tau_{1*}' \circ \tau_{2*}'$ has spectral radius equal to $49+20\sqrt{6}$. 
\end{proof}

\begin{rem}
We chose to describe the example on a blow-up of a cubic threefold since this is the general setting of this paper.
However one can do essentially the same construction starting with a curve of genus $2$  and degree~$8$ on $\p^3$, whose blow-up is studied in \cite{BL}. 
\end{rem}

\bibliographystyle{myalpha}
\bibliography{biblio}

\end{document}